\documentclass[letterpaper, 10 pt, journal, twoside]{ieeetran} 

\IEEEoverridecommandlockouts  
                                                          
\newtheorem{theorem}{Theorem}[section]
\newtheorem{lemma}[theorem]{Lemma}
\newtheorem{definition}[theorem]{Definition}
\newtheorem{corollary}[theorem]{Corollary}
\newtheorem{proposition}[theorem]{Proposition}
\newtheorem{problem}{Problem}
\newtheorem{remark}{Remark}
\newtheorem{assumption}{Assumption}
\newtheorem{example}{Example}

\newenvironment{proof}{\vspace{1ex}\noindent{\itshape 
	Proof:}\hspace{0.5em}} {\hfill\QEDBL\vspace{1ex}}

\newcommand{\mc}{\mathcal}
\newcommand{\mb}{\mathbb}

\newcommand{\rank}{\operatorname{rank}}
\newcommand{\real}{\mathbb{R}} 
\newcommand{\integ}{\mathbb{Z}}

\newcommand{\realpos}{\mathbb{R}_{\geq 0}}
\newcommand{\integpos}{\mathbb{Z}_{\geq 0}}

\newcommand{\tsp}{\mathsf{T}} 
\newcommand{\pinv}{\dagger} 
\newcommand{\inv}{{\negat 1}} 
\newcommand{\negat}{\scalebox{0.75}[.9]{\( - \)}}
\newcommand*{\QEDB}{\hfill\ensuremath{\square}}
\newcommand*{\QEDBL}{\hfill\ensuremath{\blacksquare}}

\usepackage{relsize}

\newcommand{\map}[3]{#1: #2 \rightarrow #3}
\newcommand{\setdef}[2]{\{#1 \; : \; #2\}}
\newcommand{\sbs}[2]{{#1}_{\textup{#2}}}
\newcommand{\sps}[2]{{#1}^{\textup{#2}}}
\newcommand{\until}[1]{\{1,\dots,#1\}}
\newcommand{\norm}[1]{\Vert #1 \Vert}

\newcommand{\smallsim}{\text{\larger[-2]$\sim$}}
\DeclareMathAlphabet{\mymathbb}{U}{BOONDOX-ds}{m}{n}
\newcommand{\one}{\mathds{1}}
\newcommand{\zero}{\mymathbb{0}}
\newcommand{\E}[2]{\mathbb{E}_{#1}\left[ #2 \right]}

\usepackage{hyperref}
\usepackage{graphicx,color}
\graphicspath{{./IMG-JOURNAL/}}
\usepackage{amsmath}
\usepackage{amssymb}
\usepackage{mathrsfs}
\usepackage{subfigure}
\usepackage{url}
\usepackage{booktabs}
\usepackage{array}
\usepackage[table]{xcolor}
\usepackage{mathtools} 		
\usepackage{epstopdf}
\usepackage{cite}
\usepackage[vlined,ruled]{algorithm2e}
\usepackage{upgreek}
\usepackage{dsfont}

\usepackage{epstopdf}
\title{\LARGE \textbf{%
Online Stochastic Optimization for Unknown Linear Systems: Data-Driven  Synthesis and Controller Analysis
}\thanks{A preliminary version of this paper will appear at the 2021 IEEE Conference on Decision and Control as~\cite{GB-MV-JC-ED:21-cdc}. This work was supported by the National Science Foundation through Awards CMMI 2044946 and 2044900, and by  the  National  Renewable  Energy  Laboratory  through  the subcontract UGA-0-41026-148.}}

\author{Gianluca Bianchin, Miguel Vaquero, Jorge Cort\'{e}s, and  Emiliano Dall'Anese\thanks{
 G. Bianchin and E. Dall'Anese are with the Department of Electrical, Computer, and Energy Engineering, University of Colorado Boulder. M. Vaquero is with the School of Human Sciences and Technology, IE University. J. Cort\'es is with the Department of Mechanical and Aerospace Engineering, University of California San Diego.
}}

\begin{document}
\maketitle

\begin{abstract}
This paper proposes a data-driven control framework to regulate an 
unknown, stochastic linear dynamical system to the solution of a 
(stochastic) convex optimization problem.
Despite the centrality of this problem, most of the available methods 
critically rely on a precise knowledge of the system dynamics (thus requiring 
off-line system identification and model refinement).  To this aim, in 
this paper we first show that the steady-state transfer function of a 
linear system can be computed directly from control experiments, 
bypassing explicit model identification. Then, we leverage the 
estimated transfer function to design a controller -- which is 
inspired by stochastic gradient descent methods --  that regulates 
the system to the solution of the prescribed optimization problem.
A distinguishing feature of our methods is that they do not require 
any  knowledge of the system dynamics, disturbance terms, or their 
distributions. Our technical analysis combines concepts and tools from
behavioral system theory, stochastic  optimization with 
decision-dependent distributions, and stability analysis. We 
illustrate the applicability of the framework on a case study for 
mobility-on-demand ride service  scheduling in Manhattan, NY.
\end{abstract}

\section{Introduction}
This paper focuses on the design of output feedback controllers to 
regulate the inputs and outputs of a discrete-time linear 
time-invariant system to the solution of a convex optimization 
problem. Our controller synthesis is inspired by principled 
optimization methods, properly modified to account for output feedback 
from the dynamical system (similarly to~\cite{AJ-ML-PB:09,FB-HD-CE:12,MC-ED-AB:20,LL-ZN-EM-JS:18,AH-SB-GH-FD:20,GB-JC-JP-ED:21-tcns,MN-MM:19,GB-DL-MD-SB-JL-FD:21}). 
These problems are relevant in application domains such as power 
grids~\cite{SM-AH-SB-GH-FD:18,MC-ED-AB:20}, transportation
systems~\cite{GB-JC-JP-ED:21-tcns}, robotics~\cite{MN-MM:19}, and control of epidemics~\cite{bianchin2021can}, where the target optimization problem encodes desired performance objectives and constraints (possibly dynamic and time-varying) of the system at equilibrium. Within this broad context, we propose a new approach for the synthesis of \emph{data-driven} feedback controllers for \emph{unknown stochastic LTI dynamical systems}, and we consider the case where the system is driven towards 
optimal solutions of a \emph{stochastic optimization problem}.    

Most of the recent literature on online optimization for dynamical 
systems critically relies on the assumption that the system dynamics 
are known~\cite{AJ-ML-PB:09,FB-HD-CE:12,MC-ED-AB:20,AH-SB-GH-FD:20,GB-JC-JP-ED:21-tcns,GB-JIP-ED:20-automatica,GB-DL-MD-SB-JL-FD:21}.
Unfortunately, perfect system knowledge is rarely available in 
practice -- especially when exogenous disturbances are not observable 
and/or inputs are not persistently exciting -- because maintaining and 
refining full system model often requires ad-hoc system-identification 
phases. In lieu of system-based controller synthesis, data-driven 
controllers can be fully synthesized by leveraging data from past 
trajectories. To the best of our knowledge, the design of 
optimization-based controllers that bypass model identification is 
still lacking in the literature.

\noindent \textbf{Related Work.}
Data-driven control methods exploit the ability to express the 
trajectories of a linear system in terms of a 
sufficiently-rich single trajectory, as shown by the fundamental 
lemma~\cite{JCW-PR-IM-BDM:05}. This result, developed in the context 
of the behavioral framework, has enabled the synthesis of several 
types of controllers, including static feedback 
controllers~\cite{TMM-PR:17,CDP-PT:19,ST-SA-NR-MM:20}, model 
predictive controllers \cite{JC-JL-FD:19,JB-JK-MAM-FA:20}, 
minimum-energy control laws \cite{GB-VK-FP:19}, to solve
trajectory tracking problems \cite{LX-MT-BG-GF:21}, distributed 
control problems \cite{AA-JC:20}, and recent extensions account for 
systems with nonlinear dynamics~\cite{JB-FA:20,MG-CDP-PT:20}. 

The line of research on online convex optimization \cite{EH:16}
is also related to this work. 
Several works applied online convex optimization to control plants 
modeled as algebraic maps~\cite{bolognani2013distributed,AB-ED-AS:19,CC-MC-JC-ED:19} (corresponding to cases where the dynamics 
are infinitely fast). 
When the dynamics are non-negligible,  LTI systems are 
considered in~\cite{MC-ED-AB:20,SM-AH-SB-GH-FD:18,GB-JC-JP-ED:21-tcns,%
LL-ZN-EM-JS:18}, stable nonlinear systems 
in~\cite{AH-SB-GH-FD:20,DL-DF:SM:21}, switching systems 
in~\cite{GB-JIP-ED:20-automatica}, and distributed multi-agent systems 
in~\cite{KH-JH-KU:14,FB-HD-CE:12}. 
All these works consider continuous-time dynamics and deterministic 
optimization problems, and derive results in terms of asymptotic or 
exponential stability. In contrast, here we focus on discrete-time 
stochastic LTI systems and stochastic optimization problems. 
Discrete dynamics were the focus of~\cite{MN-MM:19}, which is however 
limited to absence of disturbances.
Data-driven implementations of online optimization controllers have 
not been explored yet. 
A notable exception is~\cite{MN-MM:21}, which however does not account 
for the presence of noise, and results are limited to regret analysis. 
In contrast, when the disturbance terms are unknown, the distributions 
of the random variables that characterize the cost are parametrized by 
the decision variables, thus leading to a stochastic optimization 
problem with decision-dependent distributions, as studied 
in~\cite{JP-TZ-CM-MH:21,CM-JP-TZ-MH:20,DD-LX:20}. In this work, we 
build upon this class of problems, but accounting for two additional 
complexities: the online nature of the optimization method and the 
coupling with a dynamical system.

\noindent \textbf{Contributions.}
The contribution of this work is fourfold.
First, we show that the (steady-state) transfer function of a linear 
system can be computed from (non steady-state), finite-length 
input-output trajectories generated by the open-loop dynamics, without  
any knowledge or estimation of the system parameters. 
Second, we demonstrate that when the system is affected by unknown 
disturbances terms, the proposed data-driven framework can still be 
used to determine an approximate transfer function and, in this case, 
we explicitly characterize the approximation error.
A distinctive feature of our framework, and in contrast 
with~\cite{MY-AI-RS:20,HJVW-MKC-MM:20,AB-CD-PT:21,LX-MT-BG-GF:21},
is that we account for the presence of disturbances
affecting both the output equation and the model equation.
Third, we leverage this data-driven representation to propose a 
control method to regulate the system to an equilibrium point that 
is the solution of a stochastic optimization problem. Our design 
approach is  inspired by an online version of the stochastic projected 
gradient-descent algorithm \cite{SB:14}. One fundamental  challenge in 
the design of the  controller is that an error in the 
estimated transfer function leads to a stochastic optimization problem 
with decision-dependent distribution~\cite{DD-LX:20}, that is, the 
optimization variable induces a shift in the distribution of the 
underlying random variables that parametrize the cost. 
This is a class of problems whose direct  solution  is  intractable  
in  general \cite{JP-TZ-CM-MH:21}. To bypass this hurdle, we leverage 
the notion of \emph{stable optimizer}~\cite{JP-TZ-CM-MH:21,CM-JP-TZ-MH:20}, and we develop a stochastic controller that regulates the 
system towards such optimizer, up to an asymptotic error that depends 
on the time-variability of the optimization problem.
We show that the controller exhibits strict contractivity with respect to the stable optimizers in expectation, 
and  we explicitly quantify its transient performance.
Fourth, we study a real-time fleet management problem, where a ride 
service provider seeks to maximize its profit by dispatching its fleet 
while serving ride requests from its customers. We 
demonstrate the applicability and benefits of our methods numerically 
on a real network and demand data.

This paper generalizes the preliminary work \cite{GB-MV-JC-ED:21-cdc}
in several directions. First, we focus on optimization problems that 
are stochastic rather than deterministic. This fact raises new 
challenges in the development of optimization methods that account
for decision-dependent distributions. Second, we account for the 
presence of disturbances in the training  data. As an additional 
contribution, our treatment only requires the system to be observable,
instead of relying on direct state measurements. Fourth, we provide 
explicit (exponential) contraction bounds for the proposed control 
methods. Finally, here we illustrate the applicability of 
the methods to a ride-service scheduling problem for 
mobility-on-demand management.

\noindent \textbf{Organization.} 
The paper is organized as follows. Section~\ref{sec:2} presents some 
basic notions used in our work, in Section~\ref{sec:3} we formalize the 
problem of interest, Section~\ref{sec:4} illustrates our controller 
synthesis method, in Section~\ref{sec:5} we discuss data-driven 
techniques to compute transfer function of linear systems, 
and Section~\ref{sec:6} presents the controller analysis.
Section~\ref{sec:7} illustrates an application of the method to ride 
service scheduling and Section \ref{sec:8} concludes the paper.

\section{Preliminaries}
\label{sec:2}
In this section, we outline the notation and introduce some preliminary
concepts used throughout the paper.

\noindent \textbf{Notation.}
Given a symmetric matrix $M \in \real^{n \times n}$, $\underline \lambda(M)$ and 
$\bar \lambda(M)$ denote its smallest and largest eigenvalue, 
respectively; $M \succ 0$ indicates that $M$ is positive definite and 
$\Vert M \Vert_F$ denotes the Frobenius norm.
For a vector $u \in \real^n$, we denote 
$\norm{u}$ the Euclidean norm of $u$ and by $u^\top$ its transpose. For vectors $u \in \real^n, w \in \real^m$, 
we use the short-hand notation $(u,w) \in \real^{n+m}$ for their vector concatenation, i.e., $(u^\top, w^\top)^\top$.

%


\noindent \textbf{Persistency of Excitation.}
We next recall some useful facts on behavioral system theory 
from~\cite{JCW-PR-IM-BDM:05}. 
For a signal $k \mapsto z_k \in \real^\sigma$, $k \in \integ$, we 
denote the vectorization of $z$ restricted to the interval 
$[k, k + T ]$, $T \in \integpos$, by
\begin{align*}
 z_{[k,k+T]} = (z_k, \dots, z_{k+T}).
\end{align*}
Given $z_{[0,T-1]}$, $t \le T$, and $q\leq T-t+1$, we let $Z_{t,q}$ 
denote the Hankel matrix of length $t$ associated with $z_{[0,T-1]}$:
\begin{align*}
Z_{t,q} = \begin{bmatrix}
z_0 & z_1 & \hdots & z_{q-1}\\
z_1 & z_2 & \hdots & z_q\\
\vdots & \vdots & \ddots & \vdots\\
z_{t-1} & z_t & \hdots & z_{q+t-2}
\end{bmatrix}
\in \real^{\sigma t \times q}.
\end{align*}

Moreover, we use $[Z_{t,q}]_i$, $i \in \until{t}$ to denote the $i$-th 
block-row of  $Z_{t,q}$, namely, 
$[Z_{t,q}]_i = [z_{i-1},~z_i,~\dots,~z_{i+q-2}]$.

\smallskip
\begin{definition}{\bf \textit{(Persistently Exciting 
Signal~\cite{JCW-PR-IM-BDM:05})}}
The signal $z_{[0,T-1]}$, $z_k \in \real^\sigma$ for all 
$k \in \{0, \dots , T-1\}$, is persistently exciting of order $t$ if 
$Z_{t,q}$ has full row rank $\sigma t$.\QEDB
\end{definition}
\smallskip
We note that persistence of excitation implicitly requires 
$q \geq \sigma t$ (which in turns requires $T \geq (\sigma+1) t -1$).

Consider the linear dynamical system
\begin{align}
\label{eq:auxDynamicalSystem}
x_{k+1} &= A x_k + B u_k, & y_k &= C x_k + D u_k, 
\end{align}
with $x \in \real^{n}$, 
$u \in \real^{m}$, 
$A \in \real^{n\times n}$,
$B \in \real^{n\times m}$,
$C \in \real^{p\times n}$,
$D \in \real^{p\times m}$. 
Let $\mc C_\theta := [B, AB, A^2B, \dots, A^{\theta-1}B]$ and 
$\mc O_\nu:= [C^\tsp, A^\tsp C^\tsp, \dots, (A^\tsp)^{\nu-1} C^\tsp]^\tsp$ denote the controllability and observability matrices 
of~\eqref{eq:auxDynamicalSystem}, respectively.
The system
is controllable if $\rank(\mc C_\theta) = n$ for some 
$\theta \in \integpos$, and it is 
observable if $\rank(\mc O_\nu) = n$ for some $\nu \in \integpos$. 
The smallest integers $\mu,\nu,$ that satisfy the above 
conditions are the controllability and observability indices, 
respectively. We recall the following properties of 
\eqref{eq:auxDynamicalSystem} when its inputs are persistently 
exciting. 

\smallskip
\begin{lemma}{\textit{\textbf{(Fundamental Lemma~\cite[Corollary 2]{JCW-PR-IM-BDM:05}})}}
\label{lem:fundLemmarankHankelMatrix}
Assume \eqref{eq:auxDynamicalSystem} is controllable, let 
$(u_{[0,T-1]}, y_{[0,T-1]})$, $T \in \integ_{>0}$, be an 
input-output trajectory of \eqref{eq:plantModel}. 
If $u_{[0,T-1]}$ is persistently exciting of order $n+L$, then:
\begin{align*}
\rank \begin{bmatrix} U_{L,q}\\  X_{1,q} \end{bmatrix} 
= L m + n,
\end{align*}
where $U_{L,q}$ and $X_{1,q}$ denote the Hankel matrices associated 
with $u_{[0,T-1]}$ and $x_{[0,T-1]}$, respectively. \hfill $\Box$
\end{lemma}
\smallskip

\begin{lemma}{\bf \textit{(Data characterizes Full Behavior~\cite[Theorem 1]{JCW-PR-IM-BDM:05})}}
\label{lem:fundLemmaExistenceg}
Assume \eqref{eq:auxDynamicalSystem} is controllable and observable, let 
$(u_{[0,T-1]}, y_{[0,T-1]})$, $T \in \integ_{>0}$, be an input-output 
trajectory of \eqref{eq:plantModel}. 
If $u_{[0,T-1]}$ is persistently exciting of order $n+L$, then
any pair of $L$-long signals 
$(\tilde u_{[0,L-1]}, \tilde y_{[0,L-1]})$ is an input-output 
trajectory of \eqref{eq:plantModel} if and only if there  exists
$\alpha \in \real^q$ such that
\begin{align*}
\begin{bmatrix}\tilde u_{[0,L-1]}\\ \tilde y_{[0,L-1]}\end{bmatrix}
= 
\begin{bmatrix} U_{L,q} \\ Y_{L,q}\end{bmatrix}
\alpha,
\end{align*}
where $U_{L,q}$ and $Y_{L,q}$ denote the Hankel matrices associated 
with $u_{[0,T-1]}$ and $y_{[0,T-1]}$, respectively. \hfill $\Box$
\end{lemma}
\smallskip

In words, persistently exciting signals generate output trajectories
that can be used to express any other trajectory.

\noindent \textbf{Probability Theory.}
Let $(\Omega, \mc F, P)$ be a probability space and $z$ be a random 
variable mapping this space to $(\real^d, B_\sigma(\real^m))$, where 
$B_\sigma(\real^d)$ is the Borel $\sigma$-algebra on $\real^d$. 
Let $\mc P$ be the distribution of $z$ and $\Xi \subseteq \real^m$ be 
the support of $\mc P$.
We use $z \sim \mc P$ to denote that $z$ is distributed according 
to $\mc P$, and $\mathbb{E}_{z \smallsim \mc P}[\cdot]$ to denote the 
expectation under $\mc P$.
Let $\mc M(\Xi)$ be the space of all probability distributions 
supported on $\Xi$ with finite first moment, i.e., 
$\mathbb{E}_{z \sim \mc P}[\norm{z}] = \int_\Xi \norm{z} \mc P(dz) < \infty$
for all $\mc P \in \mc M(\Xi)$.
The Wasserstein-1 metric is:
\begin{align*}
W_1(\mc P_1, \mc P_2) := 
\inf_{\Pi \in \mc H(P_1, P_2)} 
\left \{ 
\int_{\Xi^2} \norm{z_1 - z_2} \Pi(d z_1, d z_2)
\right\},
\end{align*}
where $\mc H(P_1, P_2)$ is the set of all joint distributions with 
marginals $\mc P_1$ and $\mc P_2$.
By interpreting the decision function $\Pi$ as a transportation plan 
for moving a mass distribution described by $\mc P_1$ to another one
described by $\mc P_2$, the Wasserstein distance 
$W_1(\mc P_1, \mc P_2)$ represents the cost of an optimal mass 
transportation plan, where the transportation costs is described by 
the 1-Euclidean norm.

\begin{theorem}{\bf \textit{(Kantorovich-Rubinstein, \cite{LK-SR:58})}}
\label{thm:KR}
For any pair of distributions $\mc P_1, \mc P_2 \in \mc M(\Xi)$, the 
following holds
\begin{align*}
W_1(\mc P_1, \mc P_2) = \sup_{g \in \mc L_1} 
\left\{ 
\E{z_1 \smallsim \mc P_1}{g(z_1)} - \E{z_2 \smallsim \mc P_2}{g(z_2)} 
\right\},
\end{align*}
where $\mc L_1$ denotes the space of all $1$-Lipschitz functions, 
i.e.,
\begin{align*}
\mc L_1:= \setdef{\map{g}{\Xi}{\real^m}}{\norm{g(z_1) - g(z_2)} \leq \norm{z_1-z_2}}.
\end{align*}
\end{theorem}
\smallskip


The following result is instrumental for our analysis. We provide a 
short proof for completeness.
\smallskip
\begin{lemma}{\bf \textit{(Deviation Between Expectations 
\cite[Lemma C.4]{JP-TZ-CM-MH:21})}}
\label{lem:expectationDeviation}
Let $\map{f}{\real^n}{\real^d}$ be $L$-Lipschitz continuous.
Then, for any pair of distributions $\mc P_1, \mc P_2 \in \mc M(\Xi)$,
\begin{align*}
\norm{\E{z_1 \smallsim \mc P_1}{f(z_1)} - \E{z_2 \smallsim \mc P_2}{f(z_2)}} 
\leq  L W_1(\mc P_1, \mc P_2).
\end{align*}
\end{lemma}
\smallskip
\begin{proof}
Let $v \in \real^d$ be any unit vector and let 
$g(z) := v^\tsp f(z)$.
By assumption, $g(z)$ is $L$-Lipschitz continuous and therefore:
\begin{align*}
v^\tsp 
(& \E{z_1 \smallsim \mc P_1}{f(z_1)} - \E{z_2 \smallsim \mc P_2}{f(z_2)})\\
&\quad\quad = \E{z_1 \smallsim \mc P_1}{v^\tsp f(z_1)} - \E{z_2 \smallsim \mc P_2}{v^\tsp f(z_2)})\\
&\quad\quad= 
\E{z \smallsim \mc{Z}(\alpha)}{g(z)} - 
\E{z \smallsim \mc{Z}(\beta)}{g(z)}  \leq 
L W_1(\alpha, \beta),
\end{align*}
where the last inequality follows by application of 
Theorem~\ref{thm:KR}.
The result then follows by choosing 
\begin{align*}
v = \frac{\E{z_1 \smallsim \mc P_1}{f(z_1)} - \E{z_2 \smallsim \mc P_2}{f(z_2)}}{\norm{\E{z_1 \smallsim \mc P_1}{f(z_1)} - \E{z_2 \smallsim \mc P_2}{f(z_2)}}}.
\end{align*}
\end{proof}

\section{Problem Formulation}
\label{sec:3}
In this section, we present the problem that is the focus of this work
and we discuss a tractable reformulation used for our controller 
synthesis.

\subsection{Steady-State Regulation Problem for Linear Systems}
We consider discrete-time systems with linear dynamics:
\begin{align}
\label{eq:plantModel}
x_{k+1} &=  A x_{k} + B u_{k} + E w_{k}, & 
y_k &= C x_{k} + D w_k,
\end{align}%
where $k \in \integpos$ is the time index, 
$x_k \in \real^n$ is the state, $u_k \in \real^m$ denotes the 
control decision at time $k$, $w_k \in \real^r$ is an unknown 
exogenous stochastic disturbance with unknown distribution 
$w_k \sim \mc W_k$, and $y_k \in \real^p$ is the measurable 
output. We make the following assumptions on~\eqref{eq:plantModel}.

\begin{assumption}{\textit{\textbf{(System Properties)}}}
\label{ass:stabilityPlant}
The system~\eqref{eq:plantModel} is controllable and observable. 
Moreover, the matrix $A$ is Schur stable, i.e., for any $Q \succ 0$, 
there exists $P \succ 0$ such that $A^\tsp P A - P = -Q$.~\QEDB
\end{assumption}
\smallskip

Our control objective is to regulate \eqref{eq:plantModel} to the 
solutions of the following steady-state optimization problem at every 
time~$k$:
\begin{subequations}
\label{opt:objectiveproblem}
\begin{align}
\label{opt:objectiveproblem-a}
(u^*_k, x_k^*, y_k^*) \in \arg 
\min_{\bar u, \bar x, \bar y} \;\;\; 
& \E{w_k \smallsim \mc W_k}{\phi(\bar u, \bar y)}\\
\label{opt:objectiveproblem-b}
\text{s.t.} \;\;\;  & 
\bar x  = A \bar x  + B \bar u + E w_k,\\
\label{opt:objectiveproblem-c}
&\bar y = C \bar x + D w_k,
\end{align}
\end{subequations}
where $\map{\phi}{\real^m \times \real^p}{\real}$ denotes a cost
function that models losses associated with the control inputs and 
system outputs. 
Problem \eqref{opt:objectiveproblem} formalizes an 
equilibrium-selection problem, where the objective is to select an 
optimal input-state-output triple  $(u^*_k, x_k^*, y_k^*)$ that 
minimizes the expected cost specified by~$\phi$. We note that the 
optimization is time-varying because the distribution of $w_k$ is 
time-varying. 

\smallskip
\begin{remark}{\bf \textit{(Relationship with Classical Output Regulation Problem)}}
We note that although \eqref{opt:objectiveproblem} formalizes an 
optimal regulation problem with steady-state constraints similar to
the well-established output-regulation problems~\cite{ED:76}, with 
respect to the classical framework in our setting the optimal 
trajectories are not generated by an exosystems (i.e., a known 
autonomous linear model) but instead are specified as the solution of 
an optimization problem.
\QEDB
\end{remark}

We impose the following regularity assumptions on $\phi(u,y)$.

\begin{assumption}{\bf \!\textit{(Lipschitz and Convexity of Cost Function)\!}}
\label{ass:lipschitz-convexity}
\begin{enumerate}
\item[(a)] For any fixed $u \in \real^m$, the map $y \mapsto \phi(u,y)$ is $\ell$-Lipschitz continuous.
\item[(b)] For any fixed $y \in \real^p$, the map 
$u \mapsto \nabla \phi(u,y)$ is $\ell_u^\nabla$-Lipschitz 
continuous, and for any fixed $u \in \real^m$ the map
$y \mapsto \nabla \phi(u,y)$ is $\ell_y^\nabla$-Lipschitz 
continuous.
\item[(c)] For any fixed $y \in \real^p$, the map $u \mapsto \phi(u,y)$ is $\mu$-strongly 
convex, i.e., there exists $\mu \in \real_{>0}$ such that,
for all $u, u' \in \real^{m}$,
$\phi(u,y) \geq  \phi(u',y) + \nabla \phi(u',y)^\tsp ((u,y)-(u',y)) + \frac{\mu}{2} \norm{u-u'}$.
\QEDB
\end{enumerate}
\end{assumption}
\medskip


Strong convexity and Lipschitz-type assumptions impose basic conditions 
on the growth of the cost function often used for the analysis of 
first-order optimization methods~\cite{SB-LV:04}. Under strong 
convexity, the function $\phi$ 
admits a unique critical point $(u^*_k, y_k^*)$ that is also a global 
optimizer (see e.g.~\cite{HK-JN-MS:16}). Notice that, under Assumption 
\ref{ass:stabilityPlant}, $x_k^* = (A-I)^\inv (B u_k^* + E w_k)$ is
also unique since the linear map $(A-I)^\inv$ has an empty null space.
Hence, in what follows we use $(u^*_k, x_k^*, y_k^*)$ to denote 
the unique solution of~\eqref{opt:objectiveproblem}.
We formalize the problem that is the focus of this work next.

\smallskip
\begin{problem}
\label{prob:1}
Design an output-feedback controller of the form 
$u_{k+1} = \mathcal{C}(u_k,y_k)$ such that, without any prior 
knowledge of the matrices $(A,B,C,D,E)$ as well as of the 
noise distributions~$\mc W_k$, the input and output of 
\eqref{eq:plantModel} converge asymptotically to the 
time-varying optimizer of \eqref{opt:objectiveproblem}. \QEDB
\end{problem}

\subsection{Problem Reformulation for Unknown Dynamics}
Since the optimization problem \eqref{opt:objectiveproblem} contains 
only equality constraints, Assumption \ref{ass:stabilityPlant} can be 
used to recast it as an unconstrained optimization problem, as 
described next.
For any fixed $u \in \real^m$ and $w \in \real^r$, 
Assumption~\ref{ass:stabilityPlant} guarantees that 
\eqref{eq:plantModel} has a unique exponentially-stable equilibrium 
point $x =(I-A)^\inv (B u+E w)$.
At the equilibrium, the dependence between system inputs and 
outputs is given by
\begin{align}
\label{eq:yTransferFunctions}
y = \underbrace{C (I-A)^\inv B}_{:=G} u
+ \underbrace{(D+C (I-A)^\inv E)}_{:=H} w.
\end{align}
By using the above representation, \eqref{opt:objectiveproblem} can 
be rewritten as:
\begin{align}
\label{opt:objectiveproblem_GH}
u_k^* = \arg \min_{\bar u} \;\;\; 
& \E{w_k\smallsim \mc W_k}{\phi(\bar u, G \bar u + H w_k)}.
\end{align}

Next, we observe that the optimization 
problem~\eqref{opt:objectiveproblem_GH} is parametrized by 
the matrix $G$ as well as by the term $Hw_k$, which are not known 
when the system dynamics~\eqref{eq:plantModel} are unknown. When an 
approximation $\hat G$ of $G$ is available, the optimization 
problem~\eqref{opt:objectiveproblem_GH} can be equivalently rewritten 
as:
\begin{align}
\label{opt:objectiveproblem_2}
u_k^* = \arg \min_{\bar u} \;\;\; 
& \E{\bar z \smallsim \mc{Z}_k(\bar u)}{\phi(\bar u, 
\hat G \bar u + \bar z)},
\end{align}
where $\bar z := (G - \hat G) \bar u + H w_k$ is a random variable that 
encodes the lack of knowledge of the map $G$ as well as of the term 
$H w_k$. We note that the distribution of $\bar z$ is parametrized by 
the decision variable $\bar u$ and, in order to emphasize such 
dependency, in what follows we use the notation\footnote{%
To ease the notation in what follows we denote
$\mathbb{E}_{\bar z \smallsim \mc{Z}_k(\bar u)}{[\cdot]}$ in compact 
form as $\mathbb{E}_{\mc{Z}_k(\bar u)}{[\cdot]}$, since the random 
variable with respect to which the expectation is taken is made 
clear in the argument.}
$\bar z \sim \mc Z_k(\bar u)$.
From an optimization perspective, seeking a solution 
of~\eqref{opt:objectiveproblem_2} raises three main challenges:
\begin{enumerate}

\item[(Ch1)] Because the distribution of $\bar z$ is parametrized by 
the decision variable $\bar u$, the resulting cost function is 
nonlinear in $\bar u$, thus making the optimization 
problem~\eqref{opt:objectiveproblem_2}  intractable for general costs 
(even when $u \mapsto \phi(u,z)$ is convex).

\item[(Ch2)] Since the distribution of the disturbance $w_k$ is 
unknown, the distribution $\mc Z_k(\bar u)$ is also unknown for any 
$\bar u$. Instead, we only have access to evaluations of the random 
variable $\bar z$ via measurements of the output $y_k$ of the 
system~\eqref{eq:plantModel}.
This fact calls for the development of control methods that can adjust
the input $u_k$ based on noisy evaluations of the cost function 
through access to the system output $y_k$.

\item[(Ch3)] The closed-form expression for $G$ depends on the 
system matrices $(A,B,C)$ (cf. \eqref{eq:yTransferFunctions}), which 
are unknown. This raises the question of how to construct an 
approximate map $\hat G$ and how to quantify the approximation error. 
\end{enumerate}

The subsequent sections address the above challenges. 
Precisely, (Ch1) is addressed in Section~\ref{sec:4}.A, while in 
Section~\ref{sec:4}.B we propose a control method to address (Ch2).
Section \ref{sec:4} illustrates a data-driven technique to tackle 
(Ch3), while Section~\ref{sec:6} combines the methods by analyzing the
performance of the control technique.

\section{Synthesis of Online Optimization Controllers for Unknown Linear Systems}
\label{sec:4}
In this section, we tackle challenge (Ch1) and we show that the 
optimization problem \eqref{opt:objectiveproblem_2} can be regarded as 
a standard stochastic optimization problem by accounting for worst-case 
shifts in the distribution. 
Moreover, we use techniques from online optimization to address 
challenge (Ch2).

\subsection{Notion of Stable Optimizer}
Because of the direct dependence between the random variable $\bar z$ 
and the decision variable $\bar u$ in \eqref{opt:objectiveproblem_2}, 
explicit solution of  the optimization problem 
\eqref{opt:objectiveproblem_2} are out of reach in general.
For this reason, we focus on the problem of making online control 
decisions $u_k$ such that, when $u_k$ is applied as an input to 
\eqref{eq:plantModel}, the resulting cost is optimal for the 
distribution induced on the random variable $\bar z$. This concept is 
formalized in the following definition, inspired 
by~\cite{JP-TZ-CM-MH:21}.

\begin{definition}{\bf \textit{(Stable Optimizer)}} 
\label{def:stableOptimizer}
The vector 
$\sps{u}{so}_k \in \real^m$ is a stable optimizer of 
\eqref{opt:objectiveproblem_2} at time $k \in \integpos$ if:
\begin{align}
\label{eq:stableOptimizer}
\sps{u}{so}_k = \arg \min_{\bar u} \;\;\; 
& \E{\mc{Z}_k (\sps{u}{so}_k)}{
\phi(\bar u, \hat G \bar u + \bar z)}.
\end{align}
Accordingly, we let 
$\sps{x}{so}_{k} := (I-A)^\inv (B \sps{u}{so}_{k}+E w_k)$.
\QEDB
\end{definition}
\smallskip

In words, $\sps{u}{so}_k$ is a stable optimizer if it solves the 
optimization problem that originates by fixing the distribution of 
$\bar z$ to $\mc{Z}_k (\sps{u}{so}_k)$.
Accounting for stable optimizers allows us to make the optimization 
problem~\eqref{opt:objectiveproblem_2} tractable, since standard 
gradient-descent methods can be adapted to converge to $\sps{u}{so}_k$
(as we show in Section \ref{sec:6}).
Moreover, convergence to a stable optimizer is desirable because it 
guarantees that $\sps{u}{so}_k$ is optimal for the distribution that 
it~induces.
Stable optimizers, in general, may not coincide with the optimizers 
of \eqref{opt:objectiveproblem_2}, and their existence and uniqueness 
is guaranteed under suitable technical assumptions (see 
Theorem~\ref{thm:trackingBoudn}). However, an explicit error bound can 
be derived under suitable smoothness assumptions, as shown~next.
\begin{proposition}{\bf \textit{(Optimizer Gap)}}
\label{prop:errorAtEquilibrium}
Let Assumption \ref{ass:lipschitz-convexity} be satisfied, let $u_k^*$ 
be the optimizer of \eqref{opt:objectiveproblem_2}, and let 
$\sps{u}{so}_k$ be a stable optimizer as defined in 
\eqref{eq:stableOptimizer}. Then, 
\begin{align}
\label{eq:errorSouStar}
\norm{u_k^* - \sps{u}{so}_k} \leq 
\frac{2 \ell \norm{G - \hat G}}{\mu \sbs{\sigma}{min}^2(\hat G)},
\end{align}
where $\sbs{\sigma}{min}^2(\hat G)$ denotes the smallest singular value of $\hat G$.
\end{proposition}

\begin{proof}
The proof idea is similar to \cite[Theorem 4.3]{JP-TZ-CM-MH:21}.
By recalling that  $\bar z = (G - \hat G) \bar u + w_k$, a direct 
application of Theorem \ref{thm:KR} yields:
\begin{align}
\label{eq:auxW1}
W_1(\mc{Z}_k(u), \mc{Z}_k(u')) 
\leq 
\norm{G - \hat G} \norm{u-u'},
\end{align}
for any $u, u' \in \real^m.$
Next, we denote in compact form 
$f(u, z) := \phi(u, \hat G  u  +  z)$.
By recalling the definition of $u_k^*$ and of $\sps{u}{so}_k$, we have 
that
$\mathbb{E}_{\mc{Z}(u^*_k)} f(u^*_k,z) \leq 
\mathbb{E}_{\mc{Z}(\sps{u}{so}_k)} f(\sps{u}{so}_k,z)$, which
implies:
\begin{align}
\label{eq:auxExpectationBound}
\E{\mc{Z}(\sps{u}{so}_k)}{f(u^*_k,z)}
&- \E{\mc{Z}(\sps{u}{so}_k)}{f(\sps{u}{so}_k,z)}\\
& \quad \quad \leq 
\E{\mc{Z}(\sps{u}{so}_k)}{f(u^*_k,z)}
- \E{\mc{Z}(u^*_k)}{f(u^*_k)}. \nonumber
\end{align}
First, we upper bound the right hand side of 
\eqref{eq:auxExpectationBound}. To this aim, by combining 
\eqref{eq:auxW1} with Assumption \ref{ass:lipschitz-convexity}(a)
and by application of Lemma \ref{lem:expectationDeviation} we have:
\begin{align}
\label{eq:upperRight}
\E{\mc{Z}(\sps{u}{so}_k)}{f(u^*_k,z)}
- \E{\mc{Z}(u^*_k)}{f(u^*_k,z)} 
\leq 
\ell \norm{G - \hat G} \norm{u^*_k - \sps{u}{so}_k}.
\end{align}
Second, we lower bound the left hand side of 
\eqref{eq:auxExpectationBound}. To this aim, we note that Assumption 
\ref{ass:lipschitz-convexity}(c) implies:
$f(u^*_k,z) \geq f(\sps{u}{so}_k,z) 
+ \nabla_u f(\sps{u}{so}_k,z)^\tsp (u^*_k - \sps{u}{so}_k) 
+ \frac{\mu \sbs{\sigma}{min}^2(\hat G)}{2} \norm{u^*_k-\sps{u}{so}_k}^2$
for all $z$.
Moreover, since $\sps{u}{so}_k$ is a stable optimizer, it satisfies the 
following variational inequality: 
$\mathbb{E}_{\mc{Z}(\sps{u}{so}_k)}{[\nabla_u f(\sps{u}{so}_k,z)^\tsp  
(u'-\sps{u}{so}_k)]} \geq 0$ for all $u' \in \real^m$.
By combining the above two conditions we obtain:
\begin{align}
\label{eq:lowerLeftBound}
\E{\mc{Z}(\sps{u}{so}_k)}{f(u^*_k,z)}
&- \E{\mc{Z}(\sps{u}{so}_k)}{f(\sps{u}{so}_k,z)} \geq
\frac{\mu \sbs{\sigma}{min}^2(\hat G)}{2} \norm{\sps{u}{so}_k- u^*_k}^2.
\end{align}
Finally, the claim follows by combining \eqref{eq:upperRight} with 
\eqref{eq:lowerLeftBound}.
\end{proof}

Proposition \ref{prop:errorAtEquilibrium} quantifies the error between
a stable optimizer and the 
optimizer of~\eqref{opt:objectiveproblem_2}.
The bound shows that the error grows linearly with the absolute 
error $\norm{G-\hat G}$ and with the Lipschitz continuity constant 
$\ell$, and is inversely proportional to the smallest singular value 
of $\hat G$ and to the strong-convexity constant $\mu$. Notice that, 
in this case, the singular value $\sbs{\sigma}{min}^2(\hat G)$ could 
be interpreted as a conditioning number for the strong convexity 
constant $\mu$.
Finally, we note that when $G$ is known exactly (i.e., $\hat G = G$), 
then $u_k^* = \sps{u}{so}_k$. Indeed, in this case the optimization 
problem~\eqref{opt:objectiveproblem_2} does not feature distributions 
that are decision-dependent, rather, it is a stochastic problem with  
time-varying and unknown distributions.

\subsection{Controller Synthesis Technique}
To synthesize a controller, we note that under two simplifying 
assumptions: (i) the distribution $\mc W_k$ of the disturbance $w_k$ 
is known at all times, and (ii) the dynamics of \eqref{eq:plantModel} 
are infinitely fast (i.e., $y_k =G u_k+H w_k$ holds for all 
$k \in \integpos$), then~\eqref{opt:objectiveproblem_2} simplifies to 
stochastic optimization problem.
Thus, standard optimization methods~\cite{AB-ED-AS:19} advocate for the
adoption of the following discrete update to seek a solution 
of~\eqref{opt:objectiveproblem_2}:
\begin{align}
\label{eq:gradientSteadyStateMap}
u_{k+1} = u_k - 
&\eta \mb{E} [\nabla_u \phi(u_k, G u_k + H w_k) \\
&\quad\quad\quad\quad\quad + \hat G^\tsp \nabla_y \phi(u_k,G u_k + H w_k)],
\nonumber
\end{align}
where $\eta\in \real_{>0}$ is a tunable controller gain.
However, the update~\eqref{eq:gradientSteadyStateMap}
suffers from the following two main limitations: 
(i) the update requires evaluations of the gradient functions at 
the points $G u_k + H w_k$, which are unavailable when $G, H$, and 
$w_k$ are unknown, and (ii) computing the expectation 
in~\eqref{eq:gradientSteadyStateMap} requires full knowledge of the 
distributions $\mc W_k$ for all $k \in \integpos$.
In order to overcome limitation (i), we replace the steady-state map 
$G u_k + H w_k$ with instantaneous samples of the output $y_k$ (thus 
making the algorithm \textit{online} \cite{EH:16}) and, to cope with 
limitation (ii), we replace exact gradient evaluations with samples
collected at the current time step (thus making the algorithm 
\textit{stochastic}), and we propose the following stochastic 
gradient-descent controller for \eqref{eq:plantModel}:
\begin{align}
\label{eq:gradientFlowController}
x_{k+1} &=  A x_k + B u_k + E w_k, ~~ y_k = C x_k + D w_k,\nonumber\\
u_{k+1} &= u_k - \eta 
(\nabla_u \phi(u_k, y_k) + \hat G^\tsp \nabla_y \phi(u_k,y_k)). 
\end{align}
With respect to \eqref{eq:gradientSteadyStateMap}, the updates
\eqref{eq:gradientFlowController} includes two fundamentally-new 
features: it accounts for an ``approximate'' $y_k$ that is 
the output of a system with non-negligible dynamics, and it describes a 
stochastic-gradient update, where the true gradient is replaced 
by its noisy versions obtained by sampling.

In the remainder of this paper, we focus on proving 
that a suitable choice of the controller gain $\eta$ guarantees 
convergence of~\eqref{eq:gradientFlowController} to a stable optimizer.
To this aim, the dynamics \eqref{eq:gradientFlowController} first need 
to be fully specified by characterizing the matrix $\hat G$, which 
is the focus of our next section.


\begin{remark}{\bf \textit{(Extensions to Constrained Optimization Problems and Time-Varying Cost Functions)}}
\label{rem:projections}
The proposed framework can be extended to more general settings. First, when the optimization problem \eqref{opt:objectiveproblem} includes 
convex constrains of the form $u \in \mc U$, where 
$\mc U \subseteq \real^m$ is a closed and convex set, then the 
controller \eqref{eq:gradientFlowController} can be modified to account 
for constraints as follows:
\begin{align}
\label{eq:projectedGradient}
u_{k+1} &= \Pi_{\mc U} (u_k - \eta 
(\nabla_u \phi(u_k, y_k) + \hat G^\tsp \nabla_y \phi(u_k,y_k))),
\end{align}
where $\map{\Pi_{\mc U}}{\real^m}{\mc U}$ denotes the orthogonal 
projection onto $\mc U$, namely, for any $z \in \real^m$
\begin{align*}
\Pi_{\mc U}(z) := \arg \min_{u \in \mc U} \norm{u-z}.
\end{align*}
In this case, by using the non-expansiveness property of the projection
operator~\cite{SB-LV:04}, all the conclusions drawn in the remainder of
this paper hold unchanged. 
Second, when the cost function of \eqref{opt:objectiveproblem} is 
time-varying, that is, $(u, y) \mapsto \phi(u,y)$ is generalized by 
$(u, y) \mapsto \phi_k(u,y)$ in \eqref{opt:objectiveproblem}, then all 
the results derived in the remainder of this paper also hold unchanged,
provided that Assumption~\ref{ass:lipschitz-convexity} holds uniformly 
in time. Examples include quadratic functions of the form 
$\phi_k(u,y) = \frac{1}{2} u^\top Q u + \frac{1}{2}\|y - \sps{y}{ref}_k\|^2$, where $Q \succ 0$ and 
$k \mapsto \sps{y}{ref}_k \in ]real^p$ is a given reference point for 
the system output at time $k$. 
\QEDB
\end{remark}

\section{Data-Driven Computation of the Transfer Function of Linear Systems}
\label{sec:5}

In this section, we tackle challenge (Ch3). To this end, we first show 
that the transfer function $G$ can be computed exactly by using  
input-output data generated from \eqref{eq:plantModel}, provided that 
disturbance terms affecting the sample data are known.
In the second part of the section, we relax the assumption that noise 
terms are known, and we devise a technique to approximate~$G$.

\subsection{Exact Computation of the Transfer Function}
We begin by assuming the availability of a set of historical data 
$y_{[0,T]}$ generated by \eqref{eq:plantModel} when 
$u_{[0,T]}$, $w_{[0,T]}$ are applied as inputs.
In order to state our result, we define:
\begin{align}
\label{eq:diffSignals}
\sps{y}{diff} &:= (y_1- y_0, ~ y_2-y_1,\dots,~ y_{T}-y_{T-1}),\nonumber\\
\sps{w}{diff} &:= (w_1- w_0, ~ w_2-w_1,\dots,~ w_{T}-w_{T-1}),
\end{align}
and we let $\sps{Y}{diff}_{\nu,q}$ and $\sps{W}{diff}_{\nu,q}$, 
respectively, be the associated Hankel matrices.
The following result provides a data-driven method to compute the map 
$G$ via algebraic operations.

\smallskip
\begin{theorem}{\bf \textit{(Data-Driven Characterization of 
Steady-State Transfer Function)}}
\label{thm:data-drivenTransferFunction}
Let Assumption \ref{ass:stabilityPlant} be satisfied and let 
$\nu \in \integ_{>0}$ denote the observability index of 
\eqref{eq:plantModel}.
Moreover, assume $u_{[0,T-1]}$ and $w_{[0,T-1]}$ are persistently 
exciting signals of order $n + \nu$, and let $q := T-\nu+1$.
The following holds:

\begin{itemize}
\item[(i)] There exists $M \in \real^{q \times m \nu}$ such that:
\begin{align}
\label{eq:definionMatrixM}
\sps{Y}{diff}_{\nu,q}M &= 0, &
\sps{W}{diff}_{\nu,q}M &= 0, \nonumber  \\
U_{\nu,q}M &= \one_\nu \otimes I_m, &
W_{\nu,q}M&=0,
\end{align}
where $\sps{Y}{diff}_{\nu,q}$ and $\sps{W}{diff}_{\nu,q}$ are defined 
in~\eqref{eq:diffSignals}.

\item[(ii)] For any $M \in \real^{q \times m \nu}$ that satisfies \eqref{eq:definionMatrixM}, 
the steady-state transfer function of \eqref{eq:plantModel}
equals $G=  [Y_{\nu,q}]_i M$, for any $i \in \until \nu$.
\hfill  \QEDB
\end{itemize}
\end{theorem}
\smallskip

\begin{proof}
\textit{(Proof of (i))}. Fix a $j \in \until{m}$, let
$\bar u = (e_j, e_j, \dots ) \in\real^{m\nu}$,
where $e_j \in \real^m$ denotes the $j$-th canonical 
vector, let $\bar w = (\zero_r, \zero_r, \dots ) \in\real^{r\nu}$, and 
let $\bar y = (G e_j, G e_j, \dots) \in\real^{p\nu}$.
Since $(\bar u, \bar w, \bar y)$ is an input-output trajectory of 
\eqref{eq:plantModel}, Lemma \ref{lem:fundLemmarankHankelMatrix} 
guarantees the existence of $m_j \in \real^q$ such that 
$U_{\nu,q} m_j=\bar u$, 
$W_{\nu,q} m_j=\bar w$, and $Y_{\nu,q} m_j=\bar y$. 
By iterating the above reasoning for all $j \in \until{m}$,
and by letting the $j$-th column of M be $m_j$, we obtain that
$U_{\nu,q}M = \one_\theta \otimes I_m$ and 
$W_{\nu,q}M=0$.
Moreover, since $\bar y$, and $\bar w$ are constant at all times
we conclude that $\sps{Y}{diff}_{\nu,q} M = 0$ and 
$\sps{W}{diff}_{\nu,q} M = 0$, which proves existence of $M$.

\textit{(Proof of (ii))}.
The proof of this claim builds upon the following observation.
Let $\bar U := I_m$, let
$\bar W := \zero_{r\times m}$, let
$\bar X := (I_n - A)^\inv B \bar U +(I_n-A)^\inv E \bar W$, and let
$\bar Y:= C \bar X + D \bar W$. Then, by substitution, $\bar Y$ 
satisfies:
\begin{align}
\label{eq:auxOutputMatrix}
\bar Y = G \bar U + H \bar W = G.
\end{align}
In words, this implies that, when the inputs $\bar U = I_m$ and
$\bar W = \zero_{r\times m}$ are applied to \eqref{eq:plantModel} and 
the state satisfies
$\bar X = (I_n - A)^\inv B \bar U +(I_n-A)^\inv E \bar W$, then the 
system output satisfies $\bar Y = G$, namely, it coincides with the 
steady-state transfer function $G$.

Building upon this observation, in what follows we show that
\eqref{eq:definionMatrixM} and \eqref{eq:auxOutputMatrix} are 
equivalent, in the sense described by 
Lemma~\ref{lem:fundLemmaExistenceg}.
Formally, let $M$ be any matrix that satisfies 
\eqref{eq:definionMatrixM}. By application of 
Lemma~\ref{lem:fundLemmaExistenceg},
$U_{\nu,q}M = \one_\nu \otimes I_m$ implies that the input applied to 
\eqref{eq:plantModel} is $\bar U = I_m$, and $W_{\nu,q}M=0$ implies 
that the exogenous disturbance applied to \eqref{eq:plantModel} is 
$\bar W := \zero_{r\times m}$.
Next, we show that the matrix  $\bar Y$ 
defined as $\bar Y = [Y_{\nu,q}]_i M$ for any  
$i \in \until \nu$ coincides with \eqref{eq:auxOutputMatrix}, namely, 
we will~show:
\begin{align}
\label{eq:proofYbarXbar}
\bar Y = [Y_{\nu,q}]_i M \quad \Rightarrow  \quad
\bar Y&= C \bar X,
\bar X = (I_n - A)^\inv B I_m.
\end{align}
To this aim, we let $\bar y_{ij} = [Y_{\nu,q}]_i m_j$ denote the 
$j$-th column of $\bar Y$, and we define 
$\bar x_{ij} := [X_{\nu,q}]_i m_j$. 
Notice that $\bar Y= C \bar X$ follows from 
$[Y_{\nu,q}]_i = C [X_{\nu,q}]_i$. 
Thus, we next show that $\bar X = (I_n - A)^\inv B I_m$. 
The proof is organized into two steps.

\noindent
\textit{(Step 1) Prove that $\bar y_{i,j} = \bar y_{i+1,j}$.}
By using $\sps{Y}{diff}_{\nu,q}M = 0$:
\begin{align}
\label{eq:proofYdiff=0}
0 &= 
[\sps{Y}{diff}_{\nu,q}]_i m_j \nonumber\\ 
&=
C
\begin{bmatrix}
(A-I_n) & B & E
\end{bmatrix}
\begin{bmatrix}
[X_{\nu,q}]_i\\ 
[U_{\nu,q}]_i\\ 
[W_{\nu,q}]_i
\end{bmatrix}
m_j 
+ D [\sps{W}{diff}_{\nu,q}]_i m_j
\nonumber \\
&= C(A-I_n) \bar x_{ij} + CB e_j,
\end{align}
where the last inequality follows from 
$\bar x_{ij} := [X_{\nu,q}]_i m_j$, 
$U_{\nu,q}M = \one_\nu \otimes I_m$,
$W_{\nu,q}M = 0$, and
$\sps{W}{diff}_{\nu,q} M = 0$.
Hence, we conclude that
\begin{align}
\label{eq:proof_yij}
\bar y_{ij} = C \bar x_{ij} = CA \bar x_{ij} + CB e_j.
\end{align}
Moreover, since \eqref{eq:proofYdiff=0} holds for all 
$i \in \until \nu$, Lemma \ref{lem:fundLemmaExistenceg} guarantees 
that $\bar u = (e_j,\dots,e_j) \in \real^{m \nu}$,
$\bar w = (\zero_r,\dots,\zero_r) \in \real^{r \nu}$,
$\bar x= (\bar x_{1j}, \dots , \bar x_{\nu j}) \in \real^{n \nu}$, and
$\bar y = (\bar y_{1j}, \dots , \bar y_{\nu j}) \in \real^{p \nu}$ is 
an input-state-output trajectory of the system \eqref{eq:plantModel}, 
and thus it satisfies the dynamics:
\begin{align}
\label{eq:proof_xij}
\bar x_{i+1, j} &= A\bar x_{ij} + B e_j, & 
\bar y_{ij} &= C \bar x_{ij}.
\end{align}
By combining \eqref{eq:proof_yij} with \eqref{eq:proof_xij} we 
conclude that:
\begin{align*}
\bar y_{ij} &= CA \bar x_{ij} + CB e_j = C(A \bar x_{ij} + Be_j) 
= C \bar x_{i+1,j} = \bar y_{i+1,j}.
\end{align*}

\noindent
\textit{(Step 2) Prove that 
$\bar x_{i,j} = (I- A)^\inv B e_j$.}
By combining $\bar y_{i,j} = \bar y_{i+1,j}$ with the dynamics
\eqref{eq:proof_xij} we obtain:
\begin{align*}
C A^k \begin{bmatrix}
A-I_n & B
\end{bmatrix}
\begin{bmatrix}
\bar x_{0,j} \\ e_j
\end{bmatrix}
= 0, \text{ for all } k \in \until{\nu-1}.
\end{align*}
By recalling that \eqref{eq:plantModel} is Observable 
(see Assumption \ref{ass:stabilityPlant}), the above identity implies
$(A-I) \bar x_{0,j} + B e_j = 0$ or, equivalently, 
$\bar x_{0,j} = (I-A)^\inv B e_j$.
By recalling that 
$\bar x= (\bar x_{1j}, \dots , \bar x_{\nu j}) \in \real^{n \nu}$
represents the state associated with the constant input sequences 
$\bar u = (e_j,\dots,e_j) \in \real^{m \nu}$ and 
$\bar w = (\zero_r,\dots,\zero_r) \in \real^{r \nu}$ 
(see \eqref{eq:proof_xij}), we obtain $x_{i+1,j} = x_{i,j}$ for
all $i \in \until{\nu-1}$, which implies that 
$\bar x_{i,j} = (I-A)^\inv B e_j$ holds for all $i$, thus proving  
\textit{Step 2}.
Finally, $\bar X = (I_n - A)^\inv B I_m$ follows by iterating the 
above reasoning for all $j \in \until{ m \nu}$.
\end{proof}

Theorem~\ref{thm:data-drivenTransferFunction} shows that $G$ can be 
computed from (non steady-state) sample data generated by the 
open-loop system \eqref{eq:plantModel}, and without 
knowledge of the matrices $(A,B,C)$. 
With reference to sample complexity, the result suggests that the 
length of the sample trajectory needed to compute $G$ grows linearly 
with the observability index $\nu$.
Two technical observations are in order. 
First, in general, any matrix $M$ chosen according to 
\eqref{eq:definionMatrixM} depends on the realization of $w_{[0,T-1]}$ 
and on the choice of the input $u_{[0,T-1]}$. 
Second, for any fixed $u_{[0,T-1]}$ and  $w_{[0,T-1]}$, in general, 
there exists an infinite number of choices of $M$ that satisfy 
\eqref{eq:definionMatrixM}.
Despite $M$ not being unique, 
Theorem~\ref{thm:data-drivenTransferFunction} guarantees that 
$[Y_{\nu,q}]_i M$ is unique and independent of the choice of 
$u_{[0,T-1]}$ and $w_{[0,T-1]}$.

\smallskip

\begin{remark}{\bf \textit{(Sample Complexity)}}
\label{rem:numberOfSamplesTheorem}
Theorem~\ref{thm:data-drivenTransferFunction} requires
persistence of excitation of the $T$-long signals $u_{[0,T-1]}$ and 
$w_{[0,T-1]}$. In addition, constructing the difference 
signals $\sps{y}{diff}$ and $\sps{w}{diff}$ requires the collection of 
one additional sample of the signals $y_{[0,T]}$ and $w_{[0,T]}$
(i.e., $T+1$ samples).
\QEDB
\end{remark}
\smallskip

In Theorem \ref{thm:data-drivenTransferFunction} we assume full 
knowledge of the disturbance terms $w_{[0,T-1]}$, affecting the 
training data. 
Next, we show that in the special case where $w_{[0,T-1]}$ is unknown 
but constant at all times, 
Theorem~\ref{thm:data-drivenTransferFunction} can still be used to 
determine the input-to-output map of \eqref{eq:plantModel}.
To this aim, for all $k \in \integpos$, define:
\begin{align}
\label{eq:systemDifferences}
d_k := x_{k+1}-x_k, 
r_k := y_{k+1}-y_k,
v_k := u_{k+1}-u_k.
\end{align}
By using \eqref{eq:plantModel}, the new variables follow 
the dynamical update: 
\begin{align}
\label{eq:update-d}
d_{k+1} &= A d_k + B v_k,& r_k &= C d_k.
\end{align}
The above observation is formalized next.

\smallskip
\begin{corollary}{\bf \textit{(Data-Driven Characterization of 
Steady-State Transfer Function with Constant Noise)}}
\label{cor:constantNoiseG}
Let Assumption \ref{ass:stabilityPlant} be satisfied and let 
$\nu \in \integ_{>0}$ denote the observability index of 
\eqref{eq:plantModel}.
Moreover, assume $u_{[0,T]}$ is a  persistently exciting signals of 
order $n + \nu$, and let $q := T-\nu+1$. 
If $w_k=w \in \real^r$ for all $k \in \{0, \dots, T\}$, then the 
steady-state transfer function of \eqref{eq:plantModel} equals 
$G=  [R_{\nu,q}]_i M$, for any $i \in \until \nu$, where
\begin{align}
\label{eq:definionMatrixM-noisy}
\sps{R}{diff}_{\nu,q}M=0, \qquad
V_{\nu,q}M &= \one_\nu \otimes I_m,
\end{align}
and $R_{\nu,q}$, $V_{\nu,q}$ are the Hankel matrices associated with 
the signals in \eqref{eq:systemDifferences}, and 
$\sps{R}{diff}_{\nu,q}$ is the Hankel matrix associated with 
$ [r_1- r_0, ~ r_2-r_1,~ \dots ~ r_{T}-r_{T-1}]$. 
\QEDB
\end{corollary}

\begin{proof}
For the dynamics \eqref{eq:update-d}, the steady-state transfer 
function from the input $v$ to the output $r$ is given by 
$G_{vr}=C(I-A)^\inv B$. 
Hence, a direct application of Theorem 
\ref{thm:data-drivenTransferFunction} 
to the signals generated by \eqref{eq:update-d} guarantees that 
$G_{vr}$ can be computed as $G_{vr} =  R_{1,q} M$, where $M$ is as in 
\eqref{eq:definionMatrixM-noisy}. 
Observe that, because $u_{[0,T]}$ is persistently exciting of order 
$n+1$, then $v_{[0,T-1]}$ is also persistently exciting of the same 
order (note that $u_{[0,T]}$ contains one additional sample as 
compared to $v_{[0,T-1]}$).
This is because the columns of $V_{1,q}$ are obtained by subtracting 
disjoint pairs of columns of $U_{1,q+1}$, which are linearly 
independent.
Finally, the claim follows by noting that 
$G_{vr}$ coincides with the steady-state transfer function $G$ of 
\eqref{eq:plantModel}, as defined in \eqref{eq:yTransferFunctions}.
\end{proof}
\smallskip

Corollary \ref{cor:constantNoiseG} provides a direct way to compute 
the transfer function $G$ when the training data is affected by 
constant noise. 
Notice that, the Hankel matrices $R_{1,q}, V_{1,q}$, and
$\sps{R}{diff}_{1,q}$ can be computed directly from an input-output 
trajectory of \eqref{eq:plantModel} by processing the data as 
described by  \eqref{eq:systemDifferences}.

\smallskip
\begin{remark}{\bf \textit{(Sample Complexity with Constant Noise)}}
\label{rem:numberOfSamplesCorollary}
Notice that statement of Corollary \ref{cor:constantNoiseG} requires
the availability of a $(T+1)$-long signal $u_{[0,T]}$, and of a 
$(T+2)$-long signal $y_{[0,T+1]}$ (where the additional sample is 
needed to compute the difference signal). By comparison with Remark 
\ref{rem:numberOfSamplesTheorem}, the presence of an 
unknown constant disturbance in the training data requires to collect 
one additional sample as opposed to the case where the disturbance is 
known.
\QEDB
\end{remark}
\smallskip

\begin{remark}{\bf \textit{(Numerical Accuracy)}}
While Theorem~\ref{thm:data-drivenTransferFunction} provides a way 
to compute $G$, it remains unclear whether collecting a number of 
control experiments and computing $G$ according to 
Theorem~\ref{thm:data-drivenTransferFunction} provides an advantage  
(numerically) as opposed to identifying the matrices $(A,B,C)$ and 
using the closed-form expression \eqref{eq:yTransferFunctions}.
Notice that computing $G=C (I-A)^\inv B$ involves a matrix inversion, 
which may be ill-conditioned when $(I-A)$ is close to singular.
In \figurename~\ref{fig:montecarloG}(a) we illustrate the numerical 
error between the closed-form expression \eqref{eq:yTransferFunctions} 
and the characterization in 
Theorem~\ref{thm:data-drivenTransferFunction}, for increasing system 
size $n$. 
Not surprisingly, the figure demonstrates that the error is an 
increasing function of $n$.
In \figurename~\ref{fig:montecarloG}(b) we compare the numerical 
accuracy in predicting the steady-state output by using the 
closed form expression \eqref{eq:yTransferFunctions} and by using the 
characterization in  Theorem~\ref{thm:data-drivenTransferFunction}.
In this simulation, $\sbs{y}{eq}$ is obtained by running the 
\eqref{eq:plantModel} to convergence under a constant input 
$\sbs{u}{eq} \in \real^m$.
The simulations reveal that the error in the two cases is of the same
order of magnitude, and that the model-based expression is, on 
average, more accurate than the data-driven counterpart. 
This fact can be interpreted by noting that solving for $M$ in 
Theorem~\ref{thm:data-drivenTransferFunction} requires a matrix 
inversion (of the matrices 
$\sps{Y}{diff}_{\nu,q}, U_{\nu,q}, \sps{W}{diff}_{\nu,q}$, and 
$W_{\nu,q}$) of Hankel matrices that have size strictly larger than 
$n$, thus possibly originating higher numerical inaccuracies as 
opposed to computing the inverse of $(A-I)$.
\QEDB
\end{remark}

\begin{figure}[t]
\centering \subfigure[]{\includegraphics[width=.9\columnwidth]{%
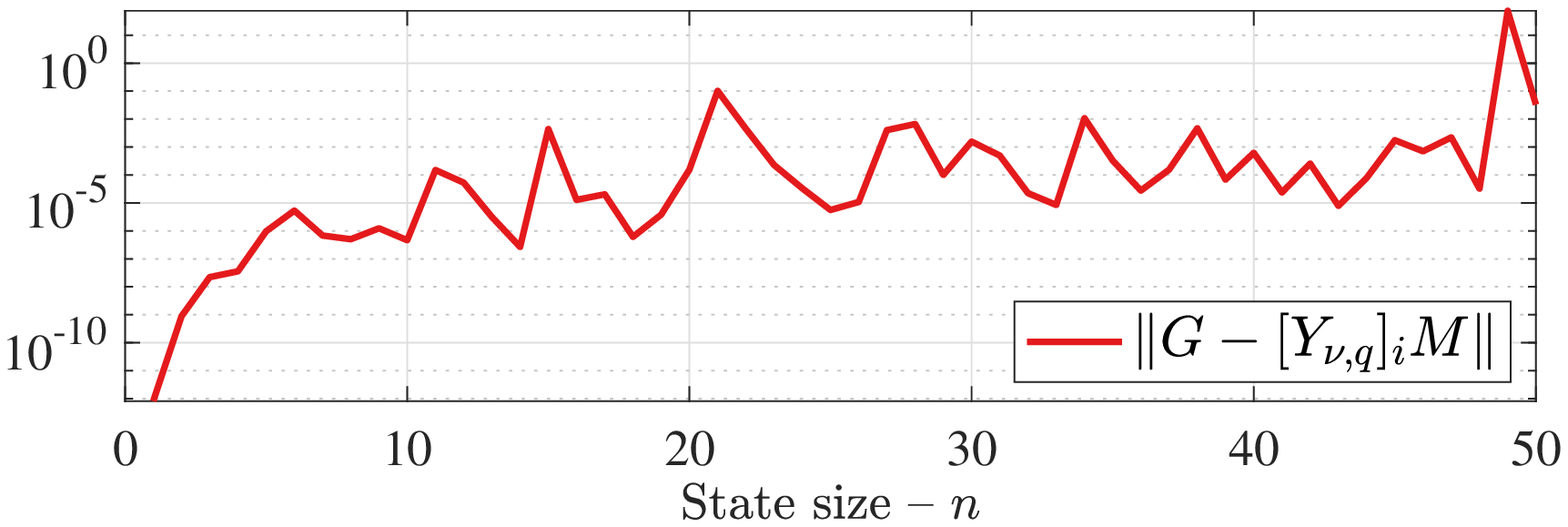}}\\
\centering \subfigure[]{\includegraphics[width=.9\columnwidth]{%
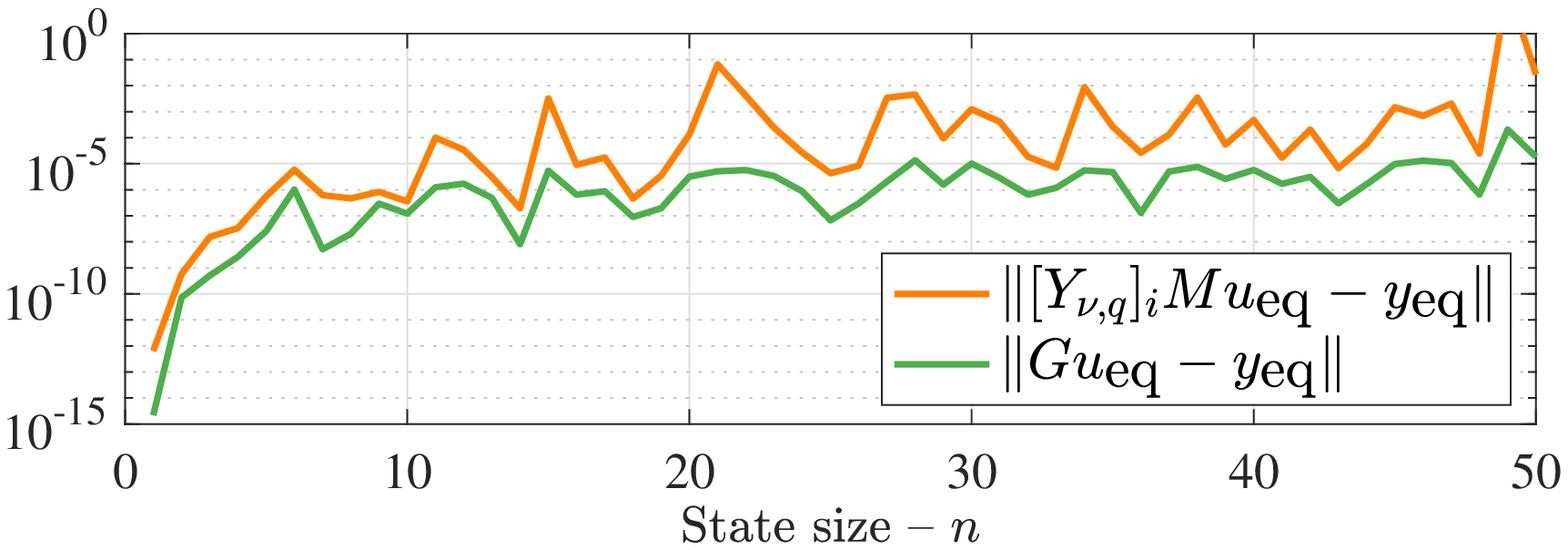}}
\caption{Montecarlo simulations illustrating precision of data-driven
transfer function computed according to 
Theorem~\ref{thm:data-drivenTransferFunction}, for increasing system 
size $n$.
The curves illustrate the average over $100$ experiments, where 
matrices $A$ and $B$ have been populated with random i.i.d. normal 
entries and where the modulus of the eigenvalues of $A$ has been 
chosen in the interval $(0,1)$. All computations have been carried out 
using the built-in function \texttt{mldivide} in 
\texttt{Matlab 2019a}.
(a) Error between closed-form expression \eqref{eq:yTransferFunctions} 
and the characterization in 
Theorem~\ref{thm:data-drivenTransferFunction} . (b) Error in predicting 
the steady-state output by using \eqref{eq:yTransferFunctions} and by 
using Theorem~\ref{thm:data-drivenTransferFunction}.
$\sbs{u}{eq}$ is chosen randomly with i.i.d. entries and $\sbs{y}{eq}$ 
is obtained by running \eqref{eq:plantModel} to convergence.}
\label{fig:montecarloG}
\end{figure}

\subsection{Unknown Noise Terms: Inexact Transfer Function}
While Theorem \ref{thm:data-drivenTransferFunction} provides a way to 
compute $G$ from data, it requires full knowledge of the disturbance 
$w_{[0,T-1]}$, which is impractical when the exogenous disturbance is 
unknown.
To this end, in the following result we characterize the error that 
originates when applying \eqref{eq:definionMatrixM} with unknown 
$w_{[0,T-1]}$.

\begin{proposition}{\bf \textit{(Error Characterization)}}
\label{prop:errorG}
Let Assumption \ref{ass:stabilityPlant} hold and let 
$\nu \in \integ_{>0}$ denote the observability index of 
\eqref{eq:plantModel}.
Moreover, assume $u_{[0,T-1]}$ and $w_{[0,T-1]}$ are persistently 
exciting signals of order $n + \nu$, and let $q := T-\nu+1$.
Assume $\hat M \in \real^{q \times m \nu}$ is any matrix that 
satisfies:
\begin{align}
\label{eq:Mhat-definition}
\sps{Y}{diff}_{\nu,q} \hat M &= 0, & 
U_{\nu,q} \hat M &= \one_\nu \otimes I_m,
\end{align}
where $\sps{Y}{diff}_{\nu,q}$ is defined according to 
\eqref{eq:diffSignals}.
If $\hat G$ is computed as $\hat G := [Y_{\nu,q}]_i \hat M$,
for any $i \in \until \nu$, then
\begin{align}
\label{eq:GhatMinusG}
 \hat G- G &= CA ([X_{\nu,q}]_i \hat M - (I-A)^\inv B )\nonumber\\
& \quad\quad + (CE+D)[W_{\nu,q}]_i\hat M +
 D [\sps{W}{diff}_{\nu,q}]_i \hat M.
\end{align}
\end{proposition}
\begin{proof}
Let $M$ be any matrix as in \eqref{eq:definionMatrixM} and 
$\hat M$ be any matrix as in \eqref{eq:Mhat-definition}.
By noting  that $\hat G - G = [Y_{\nu,q}]_i (\hat M - M)$, we will 
prove this claim by showing that $[Y_{\nu,q}]_i (\hat M - M)$ equals 
the right hand side of \eqref{eq:GhatMinusG}.
By using 
$[\sps{Y}{diff}_{\nu,q}]_i M = 0$, 
$[\sps{Y}{diff}_{\nu,q}]_i \hat M = 0$, and by recalling that 
$[\sps{Y}{diff}_{\nu,q}]_i = C(A-I) [X_{\nu,q}]_i 
+ CB [U_{\nu,q}]_i + CE [W_{\nu,q}]_i + D [\sps{W}{diff}_{\nu,q}]_i$:
\begin{align}
\label{eq:auxYMMhat}
0 &= [\sps{Y}{diff}_{\nu,q}]_i(\hat M - M)\\
&= C(A-I) [X_{\nu,q}]_i (\hat M - M) 
+ CB [U_{\nu,q}]_i (\hat M - M) \nonumber\\ 
& \quad\quad + CE [W_{\nu,q}]_i (\hat M - M)
+ D [\sps{W}{diff}_{\nu,q}]_i (\hat M - M),\nonumber\\
&= C(A-I) [X_{\nu,q}]_i (\hat M - M) 
+ CE [W_{\nu,q}]_i \hat M + D [\sps{W}{diff}_{\nu,q}]_i \hat M,
\nonumber
\end{align}
where we used $ [U_{\nu,q}]_i \hat M  =  [U_{\nu,q}]_i M $,
$ [W_{\nu,q}]_i M =0$, and $[\sps{W}{diff}_{\nu,q}]_i M=0$.
Next, by recalling that 
$[Y_{\nu,q}]_i (\hat M - M) = C [X_{\nu,q}]_i (\hat M - M) 
+ D[W_{\nu,q}]_i (\hat M - M)$
and by using \eqref{eq:auxYMMhat}:
\begin{align*}
[Y_{\nu,q}]_i (\hat M - M) &= 
CA [X_{\nu,q}]_i (\hat M - M) \\
&\quad\quad + (CE+D) [W_{\nu,q}]_i \hat M,
+ D [\sps{W}{diff}_{\nu,q}]_i \hat M.
\end{align*}
Finally, by iterating \textit{Step 2} in the proof of 
Theorem \ref{thm:data-drivenTransferFunction}, we obtain
$[X_{\nu,q}]_i  M = (I- A)^\inv B$, which proves the claim.
\end{proof}

By recalling that $(I-A)^\inv B = [X_{\nu,q}]_iM $, where $M$ is any 
matrix that satisfies \eqref{eq:definionMatrixM}, 
\eqref{eq:GhatMinusG} can equivalently be written as:
\begin{align}
\label{eq:equivalentGMinusGhat}
 \hat G- G &= CA[X_{\nu,q}]_i (\hat M - M) \nonumber\\
& \quad\quad + (CE+D)[W_{\nu,q}]_i\hat M +
 D [\sps{W}{diff}_{\nu,q}]_i \hat M.
\end{align}

Proposition \ref{prop:errorG} shows that the absolute error $\hat G-G$ 
is governed by three terms: (i) the difference 
$[X_{\nu,q}]_i \hat M - (I-A)^\inv B$, which describes the error 
between an inexact equilibrium point 
$\sbs{x}{inexact} := [X_{\nu,q}]_i \hat M u$, for some 
$u \in \real^m$, obtained by using an inexact matrix $\hat M$ and an 
exact equilibrium point $\sbs{x}{exact} := (I-A)^\inv B u$
obtained by using full knowledge of the system model,
(ii) the quantity $[W_{\nu,q}]_i\hat M$, which can be made equal to 
zero only when the training data $w_{[0,T-1]}$ is known, and, 
similarly,  (iii) the quantity $[\sps{W}{diff}_{\nu,q}]_i\hat M$, which 
can also be made equal to zero only when $w_{[0,T-1]}$ is known.
We discuss in the following remark the relationship between Theorem
\ref{thm:data-drivenTransferFunction} and Proposition \ref{prop:errorG}.

\begin{remark}{\bf \textit{(Relationship Between  Theorem \ref{thm:data-drivenTransferFunction} and Proposition \ref{prop:errorG})}}
\label{rem:whenG=Ghat}
We note that if, in addition to \eqref{eq:Mhat-definition}, $\hat M$
satisfies:
\begin{align*}
\sps{W}{diff}_{\nu,q} \hat M = 0, \text{ and } W_{\nu,q} \hat M=0, 
\end{align*}
then the following identity holds: 
$ [X_{\nu,q}]_i \hat M = (I-A)^\inv B$, and thus
$\hat G - G = 0$. Hence, in this case, we recover the characterization 
presented in Theorem  \ref{thm:data-drivenTransferFunction}.
To show that $ (I-A)^\inv B = [X_{\nu,q}]_i \hat M$, we recall 
that 
$[\sps{Y}{diff}_{\nu,q}]_i = C(A-I) [X_{\nu,q}]_i 
+ CB [U_{\nu,q}]_i + CE [W_{\nu,q}]_i + D [\sps{W}{diff}_{\nu,q}]_i$
and, by using $[\sps{Y}{diff}_{\nu,q}]_i \hat M =0$, 
$[U_{\nu,q}]_i \hat M =I$, $[W_{\nu,q}]_i \hat M =0$,  and 
$[\sps{W}{diff}_{\nu,q}]_i \hat M =0$, we have:
\begin{align*}
0 = C(A-I) [X_{\nu,q}]_i \hat M + CB.
\end{align*}
Since the system is observable and the above identity holds for all 
$i \in \until \nu$, we obtain that 
$[X_{\nu,q}]_i \hat M = (A-I)^\inv B$ (see \textit{Step 2} in the 
proof of Theorem \ref{thm:data-drivenTransferFunction}), thus proving 
the equivalence.
\QEDB
\end{remark}

It follows from Remark \ref{rem:whenG=Ghat} that, in the special case 
where the training data is noiseless (i.e., 
$W_{\nu,q}=\sps{W}{diff}_{\nu,q}=0$), Proposition \ref{prop:errorG} 
guarantees that $\hat G=G$.
Finally, we discuss in the following corollary the special case where
matrix $C$ is full column-rank.

\begin{corollary}{\bf \textit{(Error Characterization for Full 
Column-Rank $C$)}}
\label{cor:errorG}
Let Assumption \ref{ass:stabilityPlant} be satisfied and let $\hat G$
be as in \eqref{eq:GhatMinusG}.
If the Observability is $\nu=1$, then 
\begin{align}
\label{eq:GhatMinusGCorollary}
\hat G  - G &=  \Big( C (I-A)^\inv E [W_{\nu,q}]_i +D [W_{\nu,q}]_i \\
&\quad\quad\quad\quad\quad\quad
+ C (I-A)^\inv C^\pinv D [\sps{W}{diff}_{\nu,q}]_i \Big) \hat M.
\nonumber
\end{align}
\end{corollary}
\begin{proof}
When $\nu=1$, then $C$ is of full column-rank, and in this case 
\eqref{eq:auxYMMhat} implies:
\begin{align*}
[X_{\nu,q}]_i (\hat M - M) &= 
(A-I)^\inv E [W_{\nu,q}]_i \hat M  \\
&\quad\quad + (A-I)^\inv C^\pinv D [\sps{W}{diff}_{\nu,q}]_i \hat M.
\end{align*}
The claim follows by recalling that
$[Y_{\nu,q}]_i (\hat M - M) = C [X_{\nu,q}]_i (\hat M - M) 
+ D[W_{\nu,q}]_i (\hat M - M)$.
\end{proof}

We conclude this section by discussing the particular case in which
$\hat M$ is chosen as the (unique) minimum-norm solution of the set of 
equations \eqref{eq:Mhat-definition}.

\begin{example}{\bf \textit{(Optimal Selection of Training Data)}}
Let $\hat M^*$ be the minimum-norm solution of 
\eqref{eq:Mhat-definition}:
\begin{align*}
\hat M^* := \arg \min_{\hat M \in \real^{q \times m \nu}} ~~~ & 
\norm{\hat M}_F\\
\text{s.t.~~~~~~~} &
\sps{Y}{diff}_{\nu,q} \hat M = 0, 
~~U_{\nu,q} \hat M = \one_\nu \otimes I_m,
\end{align*}
where we recall that $\norm{\hat M}_F$ denotes the Frobenius norm of 
$\hat M$.
Then, $\hat M^*$ can be written as:
\begin{align*}
\hat M^* = 
\begin{bmatrix} \sps{Y}{diff}_{\nu,q} \\ 
U_{\nu,q} \end{bmatrix}^\pinv 
\begin{bmatrix} 0 \\ I \end{bmatrix}
:= \begin{bmatrix}
Y^+ ~U^+
\end{bmatrix}
\begin{bmatrix} 0 \\ \one_\nu \otimes I_m \end{bmatrix}
= U^+ (\one_\nu \otimes I_m),
\end{align*}
where $Y^+ \in \real^{q \times p\nu}$ and 
$U^+ \in \real^{q \times m\nu}$ are matrices that
satisfy the identities: 
$\sps{Y}{diff}_{\nu,q} Y^+ = I$, $U_{\nu,q} Y^+ = 0$,
$\sps{Y}{diff}_{\nu,q}  U^+ = 0$, and $U_{\nu,q} U^+ = I$.
The above equation implies 
$[W_{\nu,q}]_i \hat M = [W_{\nu,q}]_i U^+$ and
$[\sps{W}{diff}_{\nu,q}]_i \hat M = [W_{\nu,q}]_i U^+$.

By combining these relationships with \eqref{eq:GhatMinusGCorollary}, 
we have:
\begin{align*}
\norm{\hat G  - G} &\leq   
\norm{C (I-A)^\inv E + D} \norm{[W_{\nu,q}]_i \hat M} \\
&\quad\quad\quad\quad\quad\quad
+ \norm{C (I-A)^\inv C^\pinv D} 
\norm{ [\sps{W}{diff}_{\nu,q}]_i  \hat M}\\
&\leq 
\norm{C (I-A)^\inv E + D} \norm{[W_{\nu,q}]_i U^+ } \\
&\quad\quad\quad\quad\quad\quad
+ \norm{C (I-A)^\inv C^\pinv D} 
\norm{ [\sps{W}{diff}_{\nu,q}]_i U^+}.
\end{align*}
Since $U^+ $ is a right-inverse of $U_{\nu,q}$, the above bound 
suggests that $G - \hat G =0$ can be obtained when the signal 
$u_{[0,T-1]}$ is chosen so that the rows of $W_{\nu,q}$ and of 
$\sps{W}{diff}_{1,q}$ are orthogonal to the columns of the matrix 
$U^+$.
\QEDBL
\end{example}

\section{Tracking Performance in the Presence of Time-Varying Disturbances}\label{sec:6}

Having solved challenges (Ch1)-(Ch3), we are now ready to characterize the transient performance of the 
controller~\eqref{eq:gradientFlowController}. To this aim, we let
\begin{align}
\label{eq:gradientError}
e_k &:= \nabla_u \phi(u_k, y_k) 
+ \hat G^\tsp \nabla_y \phi(u_k,y_k) \nonumber\\
& \quad\quad\quad\quad -
\mb{E}_{y_k} [\nabla_u \phi(u_k, y_k) 
+ \hat G^\tsp \nabla_y \phi(u_k,y_k)],
\end{align}
denote the gradient error that originates from using a single-point 
gradient approximation based on the measurement of the output of the 
system.

\begin{remark}{\bf \textit{(Common Assumptions That Guarantee Bounded Gradient Error)}}
In what follows, we make the implicit assumption that 
the gradient error $\mathbb{E}[\norm{e_k}]$ is bounded. 
Such assumption is commonly adopted in the literature (see 
e.g.~\cite{AK-PR:20} for a thorough discussion). 
Commonly-adopted assumptions that guarantee boundedness of the gradient 
error include uniform boundedness assumptions of the form:
\begin{align*}
\mathbb{E}[\norm{e_k}] < \sigma, \text{ for all } k \in \integpos,
\end{align*}
for some $\sigma \in \real_{>0}$, or bounded variance assumptions 
of the form:
\begin{align*}
\mathbb{E} [\norm{e_k}^2] \leq \norm{
\mb{E}_{y_k} [\nabla_u \phi(u_k, y_k) 
+ \hat G^\tsp \nabla_y \phi(u_k,y_k)]}^2 + \bar \sigma^2.
\end{align*}
for some $\bar \sigma \in \realpos$.
We also notice that -- except for the use of the 1-norm as opposed to 
the 2-norm -- due to unbiasedness, uniform boundedness and bounded 
variance assumptions are equivalent.
\QEDB
\end{remark}

\begin{theorem}{\bf \textit{(Tracking of Time-Varying Stable 
Optimizer)}}
\label{thm:trackingBoudn}
Let Assumption \ref{ass:lipschitz-convexity} be 
satisfied, and let $\xi_k:=(x_{k}, u_k)$ and  
$\sps{\xi}{so}_k := (\sps{x}{so}_{k}, \sps{u}{so}_k)$, where
$\sps{u}{so}_{k}$ and $\sps{x}{so}_k$ are as in Definition 
\ref{def:stableOptimizer}. Then, for any $k \in \integpos$,
the solutions of \eqref{eq:gradientFlowController} satisfy:
\begin{align}
\label{eq:trackingError}
& \mathbb E [\norm{\xi_{k+1} - \sps{\xi}{so}_{k+1}} ]
\leq 
\beta_1 \mathbb E \left[\norm{u_{k}- \sps{u}{so}_k} \right]
+ \beta_2 \mathbb E \left[\norm{x_{k}- \sps{x}{so}_k}\right] 
\nonumber\\
& \quad
+ \gamma_1 \mathbb E [ \norm{e_k} ] 
+ \gamma_2 
\norm{\sps{u}{so}_{k+1}- \sps{u}{so}_k} 
+ \gamma_3 
\mathbb E [ \sup_{t \in \integpos} \norm{\sps{x}{so}_{k+1}- \sps{x}{so}_k}],
\end{align}
where $e_k$ defined in \eqref{eq:gradientError} and, for any  
$\kappa \in (0,1)$,
\begin{align*}
\beta_1 &= \sqrt{1 - \eta \mu} + \eta \hat \ell^\nabla \norm{G-\hat G},
\quad \hat\ell^\nabla := \ell_u^\nabla + \norm{\hat G} \ell_y^\nabla,\\
\beta_2 &= 
\sqrt{\frac{\bar \lambda(P)}{\underline \lambda(P)}
\left( 1 - (1\!-\!\kappa) 
\frac{ \underline \lambda(Q)}{\bar \lambda(P)} \right)}
+ \eta \hat \ell^\nabla \norm{C},\\
\gamma_1 & = \eta , \quad 
\gamma_2 = 1, \quad
\gamma_3 = \max\{ 
\sqrt{\frac{2\bar \lambda(P)}{ \kappa \underline \lambda(Q)}}, 
\frac{4\norm{A^\tsp P}}{\kappa \underline \lambda(Q)} \}.
\end{align*}
Moreover, if $\beta_1<1$ and $\beta_2<1$, then $\sps{u}{so}_k$ exists 
and is unique.
\end{theorem}

\begin{proof}
The proof is organized into four main steps.

\noindent \textit{(1 -- Change of Variables and Contraction Bound)}
Define the change of variables 
$\tilde x_k := x_k - \sps{x}{so}_k =
x_k - (I-A)^\inv B \sps{u}{so}_k - (I-A)^\inv E w_k$. Accordingly,
\eqref{eq:gradientFlowController} read as:
\begin{align*}
\tilde x_{k+1} &= A \tilde x_k + (x_k^* - x_{k+1}^*),\\
u_{k+1} &= u_k - \eta (\nabla_u \phi(u_k, C \tilde x_k + \hat G u_k + \bar z)\\
& \quad\quad\quad\quad\quad\quad +\hat G^\tsp \nabla_u \phi(u_k, C\tilde x_k + \hat G u_k + \bar z)).
\end{align*}
Next, we introduce the following compact notation to denote the 
algorithmic updates \eqref{eq:gradientFlowController} for 
all $\theta \in \real^m$, $u \in \real^m$, $x \in \real^n$:
\begin{align}
\label{eq:notationControllerUpdates}
F_\theta(u,x)&:=\underset{\mc Z(\theta)}{\mb E} [\nabla_u 
\phi(u, Cx + \hat G u + z) \\
& \quad\quad\quad\quad\quad\quad\quad 
+ \hat G^\tsp \nabla_y \phi(u, Cx + \hat G u + z)],\nonumber\\
\hat F(u,x)&:=\nabla_u \phi(u, Cx + \hat G u + z) \nonumber\\
& \quad\quad\quad\quad\quad\quad\quad 
+ \hat G^\tsp \nabla_y \phi(u, Cx + \hat G u + z),\nonumber\\
\mc{C}_\theta(u,x) &:= u - \eta F_\theta(u,x), \quad
\mc{\hat C}(u,x) := u - \eta \hat F(u,x). \nonumber
\end{align}
Accordingly, the left hand side of \eqref{eq:trackingError} satisfies:
\begin{align}
\label{eq:contractionTriangleInequality_1}
&\mathbb{E}[ \norm{\xi_{k+1}- \sps{\xi}{so}_{k+1}}] \leq 
\mathbb{E}[\norm{u_{k+1} - \sps{u}{so}_{k+1}}]
 + \mathbb{E}[\norm{\tilde x_{k+1}}] \\
&\quad\quad\quad
\leq \mathbb{E}[\norm{u_{k+1} - \sps{u}{so}_{k}} ]
+\norm{\sps{u}{so}_{k+1} - \sps{u}{so}_{k}}
+ \mathbb{E}[\norm{\tilde x_{k+1}}],  \nonumber
\end{align}
where we used
$\mathbb{E}[\norm{\sps{u}{so}_{k+1} - \sps{u}{so}_{k}}] = \norm{\sps{u}{so}_{k+1} - \sps{u}{so}_{k}}$ since stable optimizers are 
deterministic quantities.
Moreover, notice that:
\begin{align}
\label{eq:contractionTriangleInequality_2}
\mathbb{E} [ \norm{u_{k+1} - \sps{u}{so}_{k}} ]
&=
\mathbb{E} [\norm{\mc{\hat C} (u_k,\tilde x) 
-\mc C_{\sps{u}{so}_k}(\sps{u}{so}_k,0)} ]\nonumber\\
&\leq 
\mathbb{E}[\norm{e_k}] + \norm{\mc C_{u_k}(u_k,\tilde x) - 
\mc C_{\sps{u}{so}_k}(u_k,\tilde x_k)} \nonumber\\
&\quad
+ \norm{\mc C_{\sps{u}{so}_k}(u_k,\tilde x_k) - 
\mc C_{\sps{u}{so}_k}(u_k,0)} \nonumber\\
&\quad
+ \norm{\mc C_{\sps{u}{so}_k}(u_k,0)
-\mc C_{\sps{u}{so}_k}(\sps{u}{so}_k,0)}.
\end{align}
where we used 
$\mc{\hat C} (u_k,\tilde x) -\mc C_{u_k}(u_k,\tilde x) = e_k$ and we 
remark that the last three terms are deterministic quantities.
Linear convergence of \eqref{eq:gradientFlowController} is 
a direct consequence of three independent properties, namely 
contraction at the equilibrium, calmness to distributional shifts, and 
ease with respect to system dynamics, which we prove next.

\noindent \textit{(2 -- Calmness With Respect to Distributional 
Shifts)}
We will show:
$\norm{\mc{C}_{u_k}(u_k,\tilde x_k) 
- \mc{C}_{\sps{u}{so}_k}(u_k,\tilde x_k)} 
\leq  
\eta \hat \ell^\nabla 
\norm{G-\hat G} \norm{u_k - \sps{u}{so}_k}$.
Indeed, the following estimate holds:
\begin{align*}
\norm{\mc{C}_{u_k}(u_k,\tilde x_k) 
- \mc{C}_{\sps{u}{so}_k}(u_k,\tilde x_k)}  
&\leq \eta \hat \ell^\nabla W_1(\mc Z(u_k), \mc Z (\sps{u}{so}_k))\\
&\leq \eta \hat \ell^\nabla \norm{G-\hat G} \norm{u_k-\sps{u}{so}_k},
\end{align*}
where the first inequality follows by expanding 
\eqref{eq:notationControllerUpdates} and by using
Lemma \ref{lem:expectationDeviation}, and the second inequality follows 
from \eqref{eq:auxW1}.

\noindent \textit{(3 -- Ease With Respect to the System Dynamics)}
We will show that 
$\norm{\mc{C}_{\sps{u}{so}_k}(u_k,\tilde x_k) - 
\mc{C}_{\tilde u_k^*}(u_k,0)} 
\leq 
\eta \hat \ell^\nabla \norm{C} \norm{\tilde x_k}$.
By using Assumption \ref{ass:lipschitz-convexity}(b):
\begin{small}
\begin{align*}
& \norm{\mc{C}_{\sps{u}{so}_k}(u_k,\tilde x_k) 
- \mc{C}_{\sps{u}{so}_k}(u_k,0)}  \\
& \leq  \eta 
 \norm{\E{\mc{Z}_k(\sps{u}{so}_k)}{
 \nabla_u  \phi(u_k, C\tilde x_k \!+\! \hat G u_k \!+\! z)
- \nabla_u \phi(u_k, \hat G u_k \!+ \!z)}} \\
& ~
+ \eta\norm{\hat G^\tsp \!\E{\mc{Z}_k(\sps{u}{so}_k)} {\nabla_y \phi(u_k, C\tilde x_k \!+ \!\hat G u_k \!+\! z) - 
\nabla_y \phi(u_k, \hat G u_k\! +\! z)}}\\
&\leq \eta \hat \ell^\nabla \norm{C}\norm{\tilde x_k},
\end{align*}
\end{small}
which proves the claimed estimate.

\noindent \textit{(4 -- Contraction at the Equilibrium)}
We will show that
$\norm{\mc{C}_{\sps{u}{so}_k}(u_k,0) - 
\mc{C}_{\sps{u}{so}_k}(\sps{u}{so}_k,0) }
\leq  \sqrt{1-\eta \mu} \norm{u_k - \sps{u}{so}_k}$.
This fact follows directly from \cite[Thm 3.12]{SB:14}, and we provide 
a short proof for completeness.
By substituting \eqref{eq:notationControllerUpdates}:
\begin{small}
\begin{align*}
&\norm{\mc{C}_{\sps{u_k}{so}_k}(u,0)
-\mc{C}_{\sps{u}{so}_k}(\sps{u}{so}_k,0) }^2
= \norm{u - \eta F_{\sps{u}{so}_k}(u_k,0) - \sps{u}{so}_k}^2
\\
& ~=
\norm{u_k \!-\! \sps{u}{so}_k}^2  
- 2 \eta F_{\sps{u}{so}_k}(u_k,0)^\tsp (u_k \!-\! \sps{u}{so}_k)
+ \eta^2 \norm{F_{\sps{u}{so}_k}(u_k,0)}^2\\
& ~\leq
(1-\eta \mu) \norm{u_k \!-\! \sps{u}{so}_k}^2  
- 2\eta \Big( 
\E{\mc Z(\sps{u}{so}_k)}{\phi(u_k,C \tilde x_k +\hat G u_k +z)}\\
& \quad \quad + \E{\mc Z(\sps{u}{so}_k)}{\phi(\sps{u}{so}_k ,C \tilde x_k +\hat G \sps{u}{so}_k +z)} \Big)
+ \eta^2 \norm{F_{\sps{u}{so}_k}(u_k,0)}^2\\
&~\leq (1- \eta \mu) \norm{u_k-\sps{u}{so}_k}^2 
+ \alpha \Big(
\E{\mc Z(\sps{u}{so}_k)}{\phi(u_k,C \tilde x_k +\hat G u_k +z)}\\
&\quad\quad
- \E{\mc Z(\sps{u}{so}_k)}{\phi(\sps{u}{so}_k,C \tilde x_k +\hat G \sps{u}{so}_k +z)} \Big)\\
&\quad \leq (1- \eta \mu) \norm{u_k-\sps{u}{so}_k}^2,
\end{align*}
\end{small}
where $\alpha = 2 (\eta^2  \hat\ell^\nabla - \eta)$.
Above, the first inequality follows from 
$\phi(u_k,z_k)-\phi(u_k,\sps{z}{so}_k) \geq  \nabla \phi(u_k,\sps{z}{so}_k)^\tsp 
(u_k-\sps{u}{so}_k) + \frac{\mu}{2} \norm{u_k-\sps{u}{so}_k}^2$
(see Assumption \ref{ass:lipschitz-convexity}(c)), the second 
inequality follows from 
$\norm{F_{\sps{u}{so}_k}(u_k,0)}^2 \leq 2 \hat \ell^\nabla (\phi(u_k,z_k) - \phi(u_k,\sps{z}{so}_k))$ (see 
Assumption ~\ref{ass:lipschitz-convexity}(b)), and the 
last inequality holds because $\sps{u}{so}_k$ is a stable optimizer 
(see \eqref{eq:stableOptimizer}).

\noindent \textit{(5 -- Contraction of the Dynamical System)}
We will prove the following estimate:
\begin{align}
\label{eq:contractionDynamicalSystem}
&\mathbb{E} [ \norm{\tilde x_{k+1}} ] 
\leq 
\sqrt{\frac{\bar \lambda(P)}{\underline \lambda(P)}
\left( 1
- \frac{\underline \lambda(Q)}{4 \bar \lambda(P)}
\right)} 
\mathbb{E} [\norm{\tilde x_k}] \\
& \quad\quad 
+ \max\{ 
\sqrt{\frac{4\bar \lambda(P)}{ \underline \lambda(Q)}}, 
\frac{4\norm{A^\tsp P}}{\underline \lambda(Q)} \}
\mathbb{E} [ \sup_{t \in \integpos} 
\norm{\sps{x}{so}_{k+1}-\sps{x}{so}_k}].    \nonumber
\end{align}
In what follows, we fix the realization of the disturbance $w_k$ and 
(with a slight abuse of notation) we denote by $\tilde x_k$ the 
corresponding (deterministic) state of \eqref{eq:plantModel} and by 
$\sps{x}{so}_k$ the associated (deterministic) stable optimizer.
Let $V(x) := x^\tsp P x$ and define the set:
\begin{align*}
\Omega &:= \{ x\in \real^n : 
V(x) \leq \frac{4 \bar \lambda(P) \norm{A^\tsp P}}{
\underline \lambda(Q)} 
\sup_{t \in \integpos} \norm{\sps{x}{so}_{t+1}-\sps{x}{so}_t}\\
& \quad\quad\quad\quad
\text{ and } 
V(x) \leq \bar \lambda(P)
\sqrt{\frac{4\bar \lambda(P)}{ \underline \lambda(Q)}}
\sup_{t \in \integpos} \norm{\sps{x}{so}_{t+1}-\sps{x}{so}_t}
\}.
\end{align*}
We distinguish among two cases.

\noindent
\textit{(5 -- Case 1)} Suppose $\tilde x_k \not\in \Omega$. 
In this case, we have:
\begin{align}
\label{eq:implicationFormV}
V(\tilde x_{k+1}) - V(\tilde x_{k}) 
&\leq 
-\underline \lambda(Q) \norm{\tilde x_k}^2 
+ \bar \lambda(P) \norm{\sps{x}{so}_{k+1}-\sps{x}{so}_k}^2 \nonumber\\
& \quad\quad+ 2 \norm{A^\tsp P} 
\norm{\sps{x}{so}_{k+1}-\sps{x}{so}_k} 
\norm{\tilde x_k}\nonumber\\
& \leq 
-\frac{1}{4} \frac{\underline \lambda(Q)}{\bar \lambda(P)} 
V(\tilde x_{k}),
\end{align}
where the last inequality follows since $\tilde x_k \not \in \Omega$ 
and by using $V(\tilde x_k) \leq \bar \lambda(P) \norm{\tilde x_k}^2$.
By using $\underline \lambda(P) \norm{\tilde x_k}^2 \leq 
V(\tilde x_k) \leq 
\bar \lambda(P) \norm{\tilde x_k}^2$, \eqref{eq:implicationFormV} 
implies the following bound for the state:
\begin{align}
\label{eq:outsideOmega}
\norm{\tilde x_{k+1}}^2 \leq 
\frac{\bar \lambda(P)}{\underline \lambda(P)}
\left( 1
- \frac{\underline \lambda(Q)}{4 \bar \lambda(P)}
\right) 
\norm{\tilde x_{k}}^2.
\end{align}

\noindent
\textit{(5 -- Case 2)} Suppose $\tilde{x}_k \in \Omega$. 
In this case, we will show that $\Omega$ is forward-invariant, i.e.,  
$\tilde x_{k+1} \in \Omega$. 
By contradiction, let $\epsilon>0$ and let $k_1$ be the first instant 
such that one of the following conditions is satisfied:
\begin{align}
\label{eq:auxContradiction}    
V(\tilde x_{k_1}) &>
\bar \lambda(P) \frac{4 \norm{A^\tsp P}}{\underline \lambda(Q)}
\sup_{t \in \integpos} \norm{\sps{x}{so}_{t+1}-\sps{x}{so}_t} + \epsilon, \text{ or} \nonumber\\
V(\tilde x_{k_1}) &>
\bar \lambda(P) \sqrt{\frac{4 \bar \lambda(P)}{\underline \lambda(Q)}} 
\sup_{t \in \integpos} \norm{\sps{x}{so}_{t+1}-\sps{x}{so}_t} + \epsilon,
\end{align}
It follows by iterating \eqref{eq:implicationFormV} that 
$V(\tilde x_k)$ is strictly decreasing in a neighborhood of $k_1$.
Accordingly, there must exists $0 \leq k_0 < k_1$ such that
$V(\tilde x_{k_0}) > V(\tilde x_{k_1})$. But this contradicts the 
assumption that $k_1$ is the first instant that satisfies 
\eqref{eq:auxContradiction}. So $\Omega$ must be forward invariant.
By recalling the definition of $\Omega$, when $\tilde x_k \in \Omega$:
\begin{align}
\label{eq:insideOmega}
\norm{\tilde x_{k+1}}
\leq
\max\{ 
\sqrt{\frac{4\bar \lambda(P)}{ \underline \lambda(Q)}}, 
\frac{4\norm{A^\tsp P}}{\underline \lambda(Q)} \}
\sup_{t \in \integpos} \norm{\sps{x}{so}_{k+1}-\sps{x}{so}_k}.
\end{align}

Finally, the estimate \eqref{eq:contractionDynamicalSystem} follows by 
combining \eqref{eq:outsideOmega} and \eqref{eq:insideOmega} and by 
taking the expectation on both sides.

To conclude, \eqref{eq:trackingError} follows by substituting the 
estimates derived in the above five steps into 
\eqref{eq:contractionTriangleInequality_1}--\eqref{eq:contractionTriangleInequality_2}.
Notice that, existence and uniqueness of $\sps{u}{so}_k$ follows from 
contractivity and by application of the  Banach fixed-point theorem.
\end{proof}

\begin{remark}{\bf \textit{(Choices of $\eta$ that guarantee 
$\beta_1<1$)}}
\label{rem:feasibilityEta}
To guarantee $\beta_1<1$, the following conditions must hold
simultaneously:
\begin{align}
\label{eq:feasibilityEta}
\norm{G -\hat G} < \frac{\mu}{\hat \ell^\nabla}, 
\quad  \text{ and } \quad
\frac{2 \norm{G -\hat G} - \mu}{\norm{G -\hat G}^2} 
< \eta \leq  \frac{1}{\mu}.
\end{align}
Note that for any $\mu \in \real_{>0}$ and 
$\hat G \in \real^{p \times m}$, there exists a nonempty set of 
choices of $\eta$ that satisfy the second condition in 
\eqref{eq:feasibilityEta}.
To see this, notice that
$\frac{2 \norm{G -\hat G} - \mu}{\norm{G -\hat G}^2}  \leq  \frac{1}{\mu}$ is equivalent to 
$(\norm{G -\hat G} - \mu)^2 \geq 0$, which is satisfied for any
$\mu \in \real_{>0}$ and 
$\hat G \in \real^{p \times m}$.
While there always exists a choice of $\eta$ that guarantees 
the second condition, the first inequality in 
\eqref{eq:feasibilityEta} outlines a feasibility condition. 
Namely, when the absolute error $\norm{G -\hat G}$ is larger than 
the constant $\mu/\hat \ell^\nabla $, then the controller 
\eqref{eq:gradientFlowController} cannot guarantee 
contractivity.

To derive \eqref{eq:feasibilityEta}, notice that 
$\sqrt{1 - \eta \mu}$ admits a real-valued solution if an only if 
$\eta \leq 1/\mu $.
Moreover, the equation
$\sqrt{1 - \eta \mu} + \eta \hat \ell^\nabla \norm{G-\hat G}=1$
yields the solution $\eta = \eta_1 := 0$ and 
$\eta = \eta_2 := (2 \norm{G -\hat G} - \mu)/\norm{G -\hat G}^2$, 
where $\eta_2$ is real only if 
$\hat \ell^\nabla \norm{G -\hat G} \leq \mu$.
Accordingly, $\beta_1<1$ when $\eta< \eta_1$ or 
$\eta > \eta_2$ (see Fig. \ref{fig:rootsBeta=1}), which yields 
\eqref{eq:feasibilityEta}.
\QEDB
\end{remark}
\smallskip

\begin{remark}{\bf \textit{(Choices of $\eta$ that guarantee 
$\beta_2<1$)}}
To guarantee $\beta_2<1$ the controller $\eta$ must be chosen as:
\begin{align*}
\eta < \frac{1}{\hat \ell^\nabla \norm{C}} 
\left(
1 - \sqrt{
\frac{\bar \lambda(P)}{\underline \lambda(P)} \left(
1 - (1\!-\!\kappa)
\frac{\underline \lambda(Q)}{\bar \lambda(P)}
\right)
}
~\right).
\end{align*}
Notice that the quantity $\frac{\bar \lambda(P)}{\underline \lambda(P)} \left(
1 - (1\!-\!\kappa)
\frac{\underline \lambda(Q)}{\bar \lambda(P)}
\right)$ 
is always non-negative since $\underline \lambda(Q)/\bar \lambda(P)<1$
and it strictly smaller that $1$ if the open-loop dynamics 
\eqref{eq:plantModel} are contractive.
\QEDB
\end{remark}
\smallskip

%

\begin{figure}[t]
\centering 
\subfigure[]{\includegraphics[width=.49\columnwidth]{%
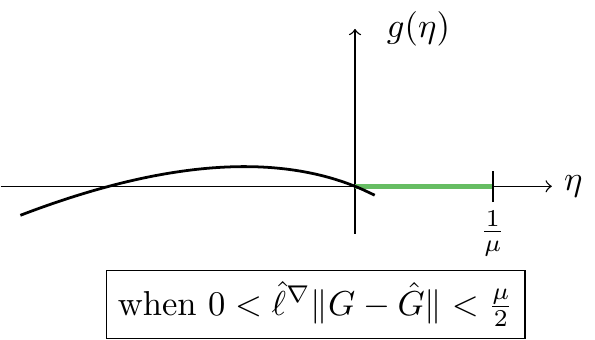}} 
\hfill
\subfigure[]{\includegraphics[width=.48\columnwidth]{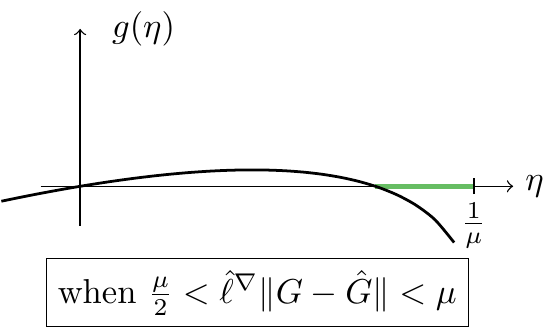}}\\
\caption{Roots of $g(\eta) = \sqrt{1 - \eta \mu} + \eta \hat \ell^\nabla \norm{G-\hat G}-1$. The green segment illustrates the set of
choices of $\eta$ that guarantee $\beta_1<1$.}
\label{fig:rootsBeta=1}
\end{figure}

Theorem~\ref{thm:trackingBoudn} provides a sufficient condition to 
guarantee that the controlled 
dynamics~\eqref{eq:gradientFlowController} converge to a stable 
optimizer $(\sps{x}{so}_k, \sps{u}{so}_k)$ (up to an asymptotic error 
that depends on the time-variability of the optimizer and on the 
sampling error). This result, combined with 
Proposition~\ref{prop:errorAtEquilibrium}, allows us to conclude convergence to a small neighborhood of the desired optimizer $(x_k^*, u_k^*)$.

According to Theorem~\ref{thm:trackingBoudn}, the rate of 
contraction of~\eqref{eq:gradientFlowController} depends
on various parameters of the optimization problem as well as of the 
dynamical system: it increases with the square root of $\eta$, 
$\mu$, and of the ratio $\underline{\lambda}(Q)/\bar \lambda(P)$ 
(that characterizes the rate of convergence of the open-loop plant 
\eqref{eq:plantModel}), and it is inversely proportional to
$\eta$, $\hat \ell^\nabla$,  $\norm{G - \hat G}$, and $\norm{C}$.
Moreover, there are three error terms that affect the bound: the 
error between the sample-based and true gradient 
$\norm{e_k}$, the shift in the stable input optimizer 
$\norm{\sps{u}{so}_{k+1}- \sps{u}{so}_k} $, 
and the worse-case shift in the stable state  optimizer 
$\sup_{t \in \integpos} \norm{\sps{x}{so}_{t+1}- \sps{x}{so}_t}$.

Some important comments on the choice of $\eta$ are in order. 
First, as discussed in Remark \ref{rem:feasibilityEta}, when 
the distributional shifts originated by the controller update are 
small (i.e., $\hat \ell^\nabla \norm{G -\hat G}<\mu/2$), then
a sufficiently-slow controller (i.e, $\eta\leq1/\mu$) guarantees 
contraction in \eqref{eq:trackingError}.
On the other hand, when the distributional shifts originated by the 
controller updates are large (i.e., 
$\hat \ell^\nabla \norm{G -\hat G}>\mu/2$), then there is a 
lower bound on the required controller gain to guarantee
contractivity 
(namely, $\eta>(2 \norm{G -\hat G} - \mu)/\norm{G -\hat G}^2$).
This fact can be interpreted by noting that a sufficiently-large 
controller gain guarantees that deviations introduced by shifts in 
the distribution (i.e. the term  
$\eta \hat \ell^\nabla \norm{G-\hat G}$) are dominated by the 
the algorithm contractivity towards the optimizer (i.e., the term 
$\sqrt{1 - \eta \mu}$).
See the proof of Theorem \ref{thm:trackingBoudn}, steps 2 and 4.
We note that this fact is in contrast with standard conditions for 
convergence of gradient-descent (see e.g. \cite{SB:14}), where 
arbitrarily-small choices of the controller gain always guarantee 
contractivity of the updates.

\begin{figure}[t]
\centering 
\subfigure[]{\includegraphics[width=.35\columnwidth]{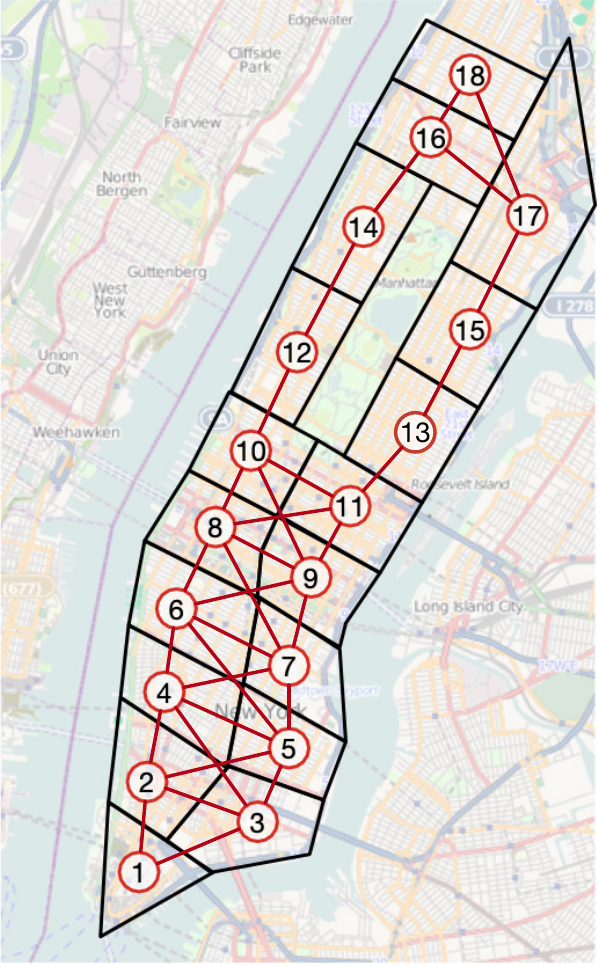}}
\hfill
\subfigure[]{\includegraphics[width=.6\columnwidth]{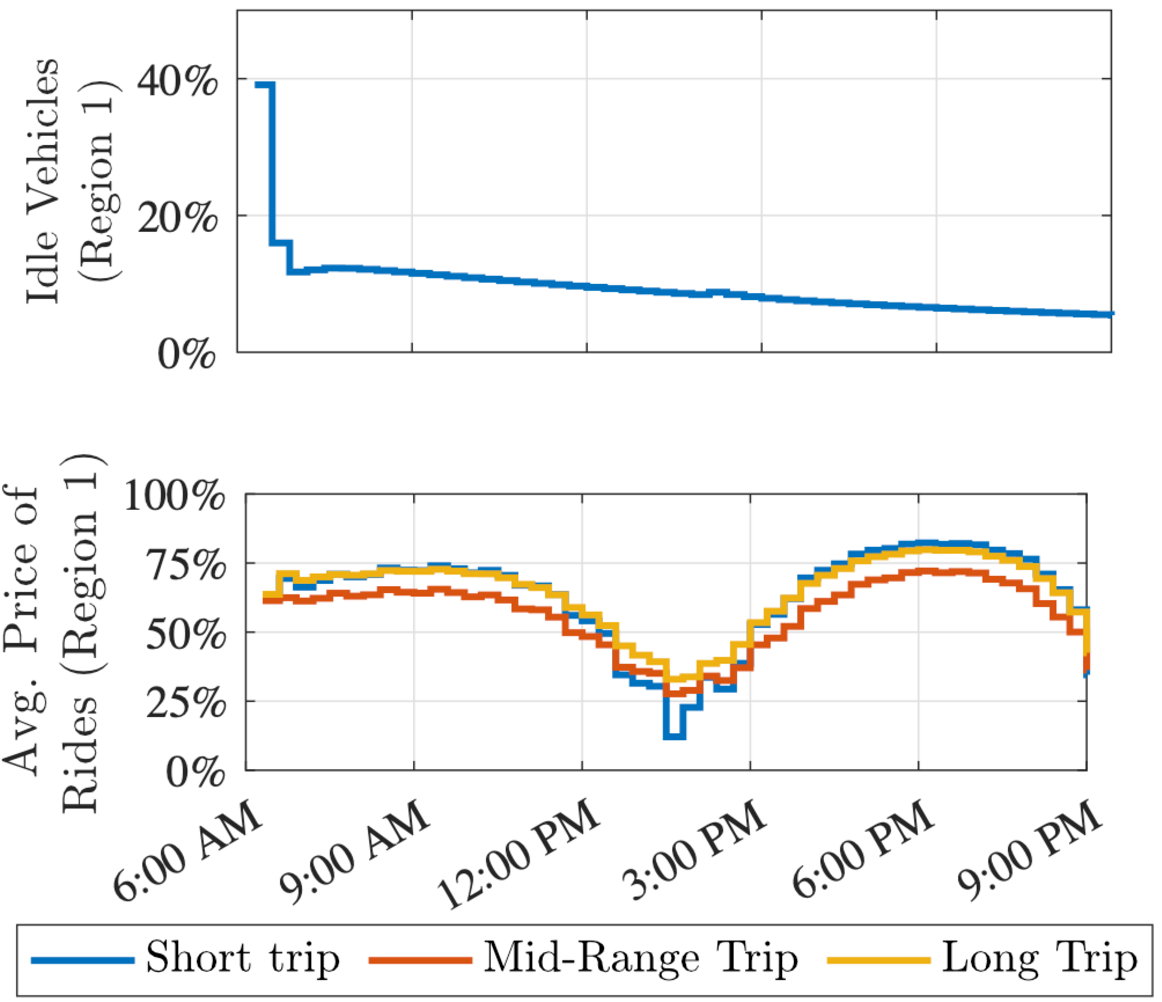}}
\caption{(a) Case study: Manhattan, NY, partitioned into $18$ regions 
of  ride requests. (b) Number of idle vehicles and cost of trips
for region 1. Short trip refers to trips from $1$ to $4$, mid range 
trip refers from $1$ to $10$, long trip refers from $1$ to $16$.
See caption of \figurename~\ref{fig:trafficOutput} for 
detailed experiment description.}
\label{fig:manhattanRegions}
\end{figure}

\section{Application to Ride-Service Scheduling}\label{sec:7}

We illustrate here the versatility and performance of the proposed controller synthesis approach in  an application scenario. A ride service provider (RSP), such as 
Uber, Lyft, or DiDi, seeks to maximize its profit by dispatching the 
vehicles in its fleet to serve ride requests from its 
customers.
We model the area of interest using a graph 
$\mc G = (\mc V, \mc E)$, where each node in $\mc V$ represents a 
region (e.g., a block or a district of a city) and an edge 
$(i,j)\in\mc E$ allows rides from node $i$ to node $j$. 
As a case study, we consider Manhattan, NY, and, similarly 
to~\cite{BT-MA:21}, we divide the area into $n=18$ region as 
in \figurename~\ref{fig:manhattanRegions}(a).
We assume that time is slotted and each slot has duration 
$\Delta = 5$ min.
We let $\delta^{ij}_k \in \realpos$ be 
the demand of rides from region $i \in \mc V$ to region 
$j\in \mc V$ at time $k \in \integpos$.
We denote by $p_{ij}^k$ the price of rides, decided 
by the RSP, from region $i$ to region $j$ at time $k$. 
We account for elasticity of the demand, whereby customers can decide 
to accept or decline rides after observing the price set by the RSP, 
and leave the system when prices are higher that the their maximum 
willingness to pay. We model the elasticity of the demand as follows:
\begin{align}
\label{eq:rideDemands}
d^{ij}_k = \delta^{ij}_k 
\Big(
1-\theta^{ij}\frac{p^{ij}_k}{\sbs{p}{max}^{ij}}
\Big).
\end{align}
Here, $d^{ij}_k$ denotes the accepted demand (after customers have 
observed the prices set by the RSP) of rides from region  $i$ to 
$j$ at time $k$, $\theta^{ij} \in [0,1]$ is a parameter that 
characterizes the steepness of elasticity, and 
$\sps{p}{max}_{ij} \in \real_{>0}$ is an upper limit on prices from
$i$ to $j$.

We let $x^i_k \in \realpos$ denote the idle-vehicle occupancy (i.e., 
the number of unoccupied vehicles, normalized by the fleet size) of 
fleet vehicles in region $i$ at time $k$.
We assume that drivers of unoccupied vehicles naturally rebalance the 
fleet, namely, they travel from regions with a high occupancy of 
(fleet) vehicles to regions with a lower occupancy in order to maximize
their profit. We denote by $a_{ij} \in \realpos$ the fraction of 
unoccupied vehicles that travel from $i$ to $j$ at every time step.
Travel times are non-negligible and can vary over time:  we model 
them by using Boolean variables:
\begin{align*}
\sigma^{ij,\tau}_{k}= 
\begin{cases}
1, & \text{if travel time from $i$ to $j$ at time $k$ is $\tau$ slots,}\\
0, & \text{otherwise},
\end{cases}
\end{align*}
defined for all $i,j \in \mc V$ and $k, \tau \in \integpos$.
In what follows, we assume that $\sigma^{ij,0}_{k}=0$ for all 
$i,j \in \mc V, k \in \integpos$, so vehicles take at least one time 
slot to travel between any pair of nodes; moreover, we let 
$\sigma^{ij,\tau}_{k}=0$ for all $\tau > T \in \integpos$, so that $T$  
is the maximum travel time in the network.
Accordingly, the occupancy of idle vehicles in each region $i$ 
satisfies the following conservation~law:
\begin{align}
\label{eq:plantApplication}
x_{k+1}^i &= x_k^i 
- \sum_{j \in \mc V} a_{ij} x_k^i 
+ \sum_{j \in \mc V} a_{ji} x_k^j\\
& \quad\quad\quad\quad\quad
- \sum_{j \in \mc V} d^{ij}_k
+ \underbrace{\sum_{j \in \mc V} \sum_{\tau=k-T}^{k-1} 
\sigma^{ji,k-\tau}_{\tau} d^{ji}_\tau + e_{k}^i}_{:= w^i_k}, \nonumber
\end{align}
In \eqref{eq:plantApplication}, the quantity 
$-\sum_{j \in \mc V} a_{ij} x_k^i$ accounts for the vehicles that 
leave the region due to fleet rebalancing, while 
$\sum_{j \in \mc V} a_{ji} x_k^j$ models rebalancing vehicles arriving 
at $i$.
The quantity $-\sum_{j \in \mc V} d^{ij}_k$ models all
customer-occupied vehicles departing $i$ at time $k$, while 
$\sum_{\tau=0}^{k} \sigma^{ji,k-\tau}_{\tau} d^{ji}_\tau$ accounts 
for occupied vehicles arriving to $i$ at time $k$.
Finally, we use the term $e_k^i$ to account for all the unmodeled 
disturbances affecting the dynamics, including inaccuracies in the 
rebalancing coefficients $a_{ij}$ and vehicles leaving or entering the 
system (e.g., drivers that stop or start driving).
Further, we assume that the travel times between regions (i.e., the 
scalars $\sigma^{ij,\tau}_{k}$) are unknown or difficult to 
estimate, and we incorporate all unknown terms in the exogenous signal 
$w_k^i$.

Because \eqref{eq:plantApplication} describes a mass conservation law, 
the dynamics \eqref{eq:plantApplication} define a compartmental model 
that is marginally stable \cite{WH-VC-EA:03}. For this reason, we define the state differences 
$\tilde x_k^i:= x_k^i - x_k^{i+1}$ for all $k \in \integpos$, 
$ i \in \mc V$. In these new variables, 
\eqref{eq:plantApplication} define a $(n-1)$-dimensional system that is
asymptotically stable and thus satisfies Assumption~\ref{ass:stabilityPlant}.

\begin{figure}[t]
\centering \subfigure[]{\includegraphics[width=.48\columnwidth]{%
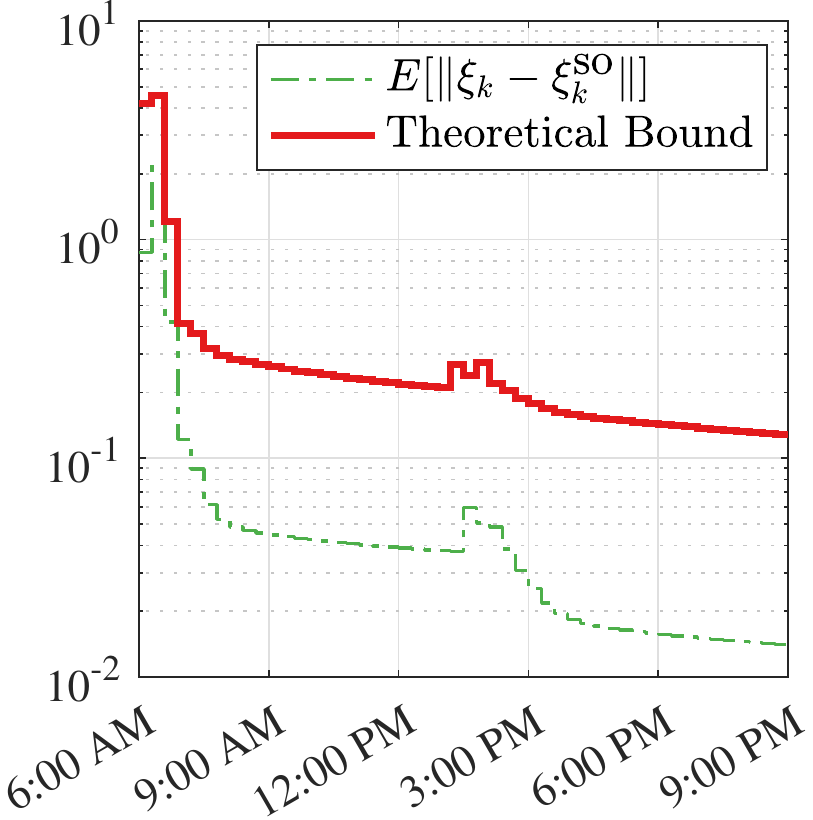}}
\centering \subfigure[]{\includegraphics[width=.48\columnwidth]{%
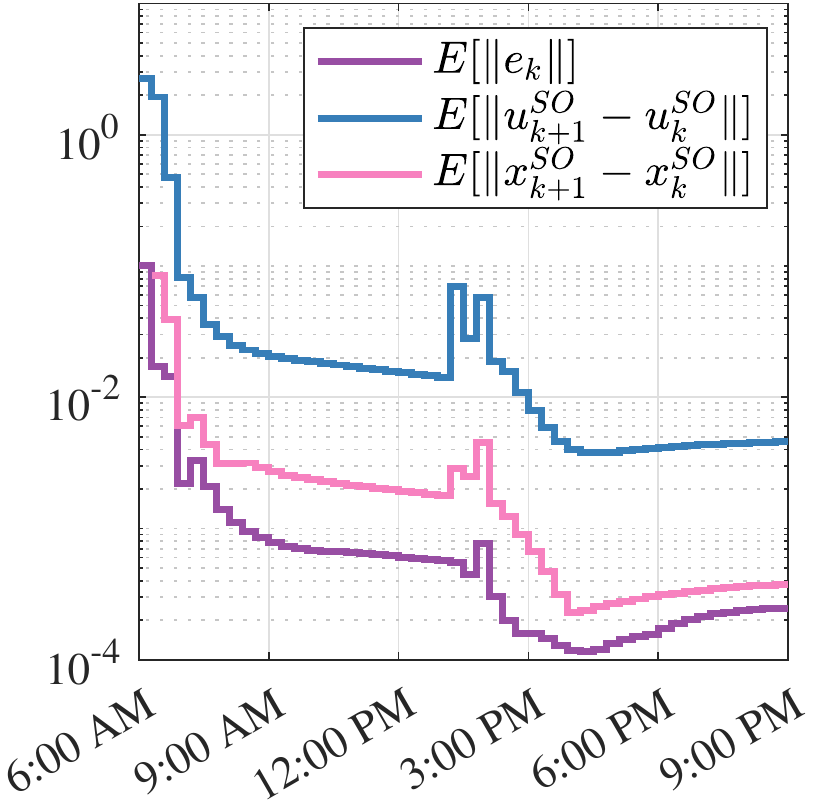}}
\caption{(a) Numerical error and error bound from Theorem~\ref{thm:trackingBoudn} and (b) terms characterizing the error
bound. Curves illustrate the average over $100$ realizations of the 
noise terms. Fluctuation at around 1:00PM is caused by a drop in ride 
demands (see \figurename~\ref{fig:trafficOutput} (top panel)).}
\label{fig:TrackingError}
\end{figure}

We formulate the RSP's objectives of selecting the price of rides in 
order to  maximize its profit as the following optimization problem to
be solved at every $k$:
\begin{align}
\label{opt:applicationProblem}
\max_{p, x, d} ~~ 
& \sum_{i \in \mc V} \sum_{j \in \mc V} 
 p^{ij}  d^{ij} - c^{ij}  d^{ij} 
- \varrho \norm{ x}^2,\nonumber\\
\text{s.t.} ~~~~ 
& 0 = - \sum_{j \in \mc V} a_{ij}  x^i 
+ \sum_{j \in \mc V} a_{ji}  x^j
- \sum_{j \in \mc V}  d^{ij} + w_k^i, \nonumber\\
&  d^{ij} = \delta^{ij}_k 
\left( 1-\theta^{ij}  p^{ij}/\sbs{p}{max}^{ij} \right),
\nonumber\\
&  d^{ij} \geq 0, \;  x^{i} \geq 0, \quad \forall i,j, \in \mc V,
\end{align}
where $ p,  x,  d$ denote the vectors obtained by stacking 
$ p^{ij}$, $ x^i$, and $ d^{ij}$, for all $i, j \in \mc V$, 
respectively, $ p^{ij}  d^{ij}$ models the RSP earnings from 
serving the demand $ d^{ij}$, the quantity
$c^{ij} d^{ij}$, $c^{ij} \in \real_{>0}$, 
models the cost of routing vehicles from $i$ to $j$, and the term
$\varrho \norm{ x}$, $\varrho \in \real_{>0}$, describes the RSP's 
objective of maximizing the fleet utilization.

\begin{figure}[t]
\includegraphics[width=.95\columnwidth]{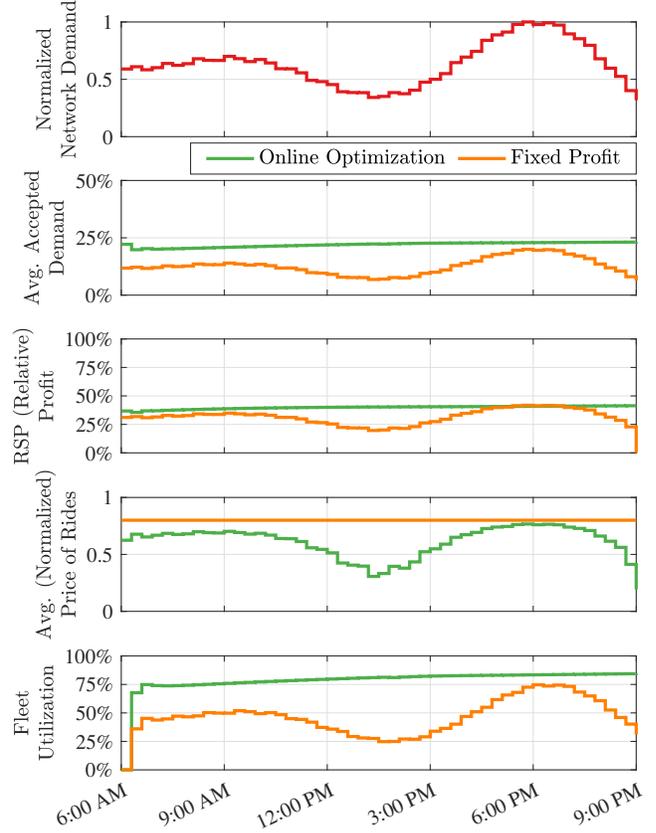}
\caption{Comparison between optimized pricing policy and fixed-pricing 
policy. In the fixed-pricing policy, the RSP sets prices to guarantee a 
25\% profit from the operational cost of the fleet. All lines show 
trajectories averaged over 100 realizations of the simulation and over 
the regions of the network (\figurename~\ref{fig:manhattanRegions}).
Network demand data was derived from \cite{TLC:21} for March 1, 2019, 
and normalized by its maximum value for illustration purposes. 
Fleet size was set to $25 \%$ of the maximum demand.
The RSP profit (third panel) was normalized by the maximum profit 
achievable if the fleet had an infinite number of vehicles.The average 
price of rides (fourth panel) was normalized by the maximum 
willingness to pay $\sps{p}{max}$.}
\label{fig:trafficOutput}
\end{figure}

To solve the optimization problems we employ the projected controller~\eqref{eq:projectedGradient} to account for constraints. 
All experiments were performed using \texttt{Matlab 2019a}, ride 
demands and locations were estimated by using the Taxi and Limousine 
Commission (TLC) data from New York City \cite{TLC:21} for March 1, 
2019 between 6:00AM and 9:00PM. Note that the available demand data 
does not describe the potential rides, but rather the realized ones. 
Although this data may not reflect the true demand, it is often used as 
a good approximation in several related works (see e.g.,~\cite{BT-MA:21}). 

\figurename~\ref{fig:TrackingError}(a) illustrates the numerical 
tracking error and the error bound characterized in Theorem 
\ref{thm:trackingBoudn}, and \figurename~\ref{fig:TrackingError}(b)
presents a breakdown of the terms characterizing the tracking error. 
The plots show that during the initial transient the tracking 
error quickly decreases, up to a steady-state value of order 
$10^{-2}$, thus validating the conclusions drawn in 
Theorem~\ref{thm:trackingBoudn}. 
\figurename~\ref{fig:TrackingError}(b) showcases that the tracking 
error does not further decrease beyond such steady-state error because 
the optimizer $(\sps{u}{so}_k, \sps{x}{so}_k)$ is changing over time.
The sudden increase in error that occurs at around 1:00PM can 
be interpreted by means of Fig 3(b), which shows that the price of 
rides, at this time of the day, must decrease consistently since the 
network experiences a drop in ride demands (see 
\figurename~\ref{fig:trafficOutput}, top panel).

In \figurename~\ref{fig:trafficOutput}, we compare the performance of 
the online optimization method with a fixed-pricing policy, whereby the 
RSP selects a fixed price for all rides, corresponding to a 25\% 
profit from the operational cost of the fleet.
A 25\% profit was selected as the maximum profit that allows the RSP to 
serve the peak of demand with the available fleet.
\figurename~\ref{fig:trafficOutput}, second panel from the top, shows 
that by using the adaptive pricing policy the RSP always accept a higher
number of rides; \figurename~\ref{fig:trafficOutput}, third panel, shows that 
the RSP profit is always higher under the adaptive pricing policy 
except at the peak of demand, which can be interpreted as an optimistic 
situation where suboptimal pricing still leads to a high utilization of 
the fleet; \figurename~\ref{fig:trafficOutput}, fourth panel shows
that the adaptive policy adjusts the price of rides based on the 
instantaneous demand,  and showcases that our optimal pricing policy 
tends to reduce the price of rides in the interest of maximizing 
fleet utilization; finally, \figurename~\ref{fig:trafficOutput}, bottom 
panel, shows that the optimized pricing policy always results in a 
higher fleet utilization.

\section{Conclusions}\label{sec:8}
We have proposed a data-driven method to design controllers that steer an unknown dynamical system to the solution trajectory of a stochastic, time-varying optimization problem. The technique does not rely on any prior knowledge or estimation of the
system matrices or the exogenous disturbances affecting the model 
equation. We have shown how knowledge of (possibly noisy) input-output data generated by the open-loop system can be used to compute the 
steady-state transfer function of the system and how to approximate it when disturbances are unknown. Our analysis has established that the resulting closed-loop dynamics are strictly 
contractive when the controller gain is chosen: (i) small enough so 
that the controller is sufficiently slower than the dynamical system and, 
simultaneously, (ii) large enough so that the controller can overcome 
shifts in the distributions associated with the lack of knowledge of 
the system dynamics.
Our work here demonstrates for the first time that online 
optimization techniques can be used to control dynamical systems even 
when the system model is unknown.
This opens up several exciting opportunities for future work, including 
extensions to scenarios where the control method guarantees persistence 
of excitation, and the generalization to scenarios with distributed 
computation, sensing, and communication.

\bibliographystyle{IEEEtran}
\bibliography{bib/alias,bib/bibliography_journal,bib/Main,bib/GB,bib/Main_GB} 

\end{document}